\documentclass[a4paper]{article}
\usepackage{amsmath,amsthm,amssymb,mathtools,cite}

\usepackage[pdftex,bookmarksopen=true,bookmarks=true,unicode,setpagesize]{hyperref}
\hypersetup{colorlinks=true,linkcolor=blue,citecolor=blue,urlcolor=blue}

\allowdisplaybreaks[3]
\numberwithin{equation}{section}

\newtheorem{theorem}{Theorem}[section]
\newtheorem{corollary}[theorem]{Corollary}
\newtheorem{lemma}[theorem]{Lemma}
\newtheorem{proposition}[theorem]{Proposition}
\theoremstyle{definition}
\newtheorem{definition}{Definition}

\theoremstyle{remark}
\newtheorem{remark}[theorem]{Remark}

\newcommand{\R}{{\mathbb{R}}}
\newcommand{\X}{{\mathbb{R}^d}}
\newcommand{\B}{{\mathcal{B}}}
\newcommand{\Bc}{{\B_\mathrm{c}}}
\newcommand{\E}{{\mathbb{E}}}
\newcommand{\N}{\mathbb{N}}
\newcommand{\eps}{\varepsilon}
\newcommand{\la}{\lambda}
\newcommand{\La}{\Lambda}
\newcommand{\ga}{\gamma}
\newcommand{\Ga}{\Gamma}
\newcommand{\kap}{{\varkappa^+}}
\newcommand{\kam}{{\varkappa^-}}
\newcommand{\kapm}{{\varkappa^\pm}}
\newcommand{\dt}{\dfrac{d}{d t}}
\newcommand{\pdt}{\dfrac{\partial}{\partial t}}
\newcommand{\sgn}{\,\mathrm{sgn}}
\newcommand{\mcr}{m_\mathrm{cr}}
\newcommand{\const}{\mathrm{const}\,}

\newcommand{\1}{1\!\!1}
\let\hat=\widehat

\title{Extinction threshold in spatial stochastic logistic model: Space homogeneous case}

\author{Dmitri Finkelshtein\thanks{Department of Mathematics,
Swansea University, Bay Campus, Fabian Way, Swansea SA1 8EN, U.K. ({\tt d.l.finkelshtein@swansea.ac.uk}).}}

\begin{document}

\maketitle

\begin{abstract}
We consider the extinction regime in the spatial stochastic logistic model in $\X$ (a.k.a. Bolker--Pacala--Dieckmann--Law model of spatial populations) using the first-order perturbation beyond the mean-field equation. In~space homogeneous case (i.e. when the density is non-spatial and the covariance is translation invariant), we show that the perturbation converges as time tends to infinity; that yields the first-order approximation for the stationary density. Next, we study the critical mortality---the smallest constant death rate which ensures the extinction of the population---as a~function of the mean-field scaling parameter $\eps>0$. We find the leading term of the asymptotic expansion (as $\eps\to0$) of the critical mortality which is apparently different for the cases $d\geq3$, $d=2$, and $d=1$.

\textbf{Keywords:} extinction threshold, spatial logistic model, mean-field equation, population density, perturbation, correlation function, asymptotic behaviour 

\textbf{2010 Mathematics Subject Classification:} 34E10, 92D25, 34D10, 58C15    
\end{abstract}

\section{Introduction}

The spatial stochastic logistic model was introduced in 1997 by B.\,Bolker and S.\,W.\,Pacala, \cite{BP1997}, and it has had a continual interest since then in both population ecology, e.g. \cite{MDL2004,LMD2003,DL2000,LD2000,OC2006}, and (pure) mathematics, e.g. \cite{FM2004,Eth2004,FKK2009,KK2016,SW2015,BKK2019,BMW2014,BMW2018}. The model describes spatial branching of individuals in a population with a density dependent death rate. We consider it in the following notations. 

We fix $m>0$, the \emph{mortality} constant, and two functions, $a_\eps^+$ and $a_\eps^-$, the \emph{dispersion} and the \emph{competition} kernels, respectively. Here $\eps>0$ is an artificial scaling parameter:
\begin{equation}\label{eq:aeps}
	a^\pm_\eps(x):=\eps^d a^\pm(\eps x), \qquad x\in\X,
\end{equation}
where $d\geq1$ and $a^\pm$ are fixed nonnegative integrable functions on $\X$, which are assumed to be non-degenerate: 
\begin{equation}\label{eq:nondegen}
\varkappa^\pm:=\int_\X a^\pm(x)\,dx=\int_\X a_\eps^\pm(x)\,dx >0, \qquad \eps>0.
\end{equation}

Let $\ga_{t,\eps}\subset\X$ denote a discrete random set of positions of individuals at a moment of time $t\geq0$. The set may be finite or locally finite (the latter means that it has a finite number of points in each compact set from~$\X$). For an infinitesimally small $\delta>0$,  there happens, with the probability $1-o(\delta)$, exactly one out of two possible events within the time-interval $[t,t+\delta)$: either the individual placed at an $x\in\ga_{t,\eps}$ sends an off-spring to an area $\Lambda\subset\X$ with the probability
\[
	\delta \,  \int_\La a_\eps^+(x-y)\,dy+o(\delta);
\]
or the individual placed at an $x\in\ga_{t,\eps}$ dies with the probability
\[
\delta  \biggl(m+\sum_{y\in\ga_{t,\eps}\setminus\{x\}} a_\eps^-(x-y)\biggr)   + o(\delta).
\]

In population ecology, one of the fundamental questions relates to the persistence of populations, or conversely to the possibility of their extinction. The~latter can be defined 
through the equation
\begin{equation}\label{eq:extinction}
\lim_{t\to\infty}k_{t,\eps}(x)=0, \quad x\in\X
\end{equation}
(we write henceforth $x\in\X$ instead of `for a.a. $x\in\X$'),
where $k_{t,\eps}(x)\geq0$ denotes the local population density given through the equality
\begin{equation}\label{eq:kt1def}
\E \bigl[|\ga_{t,\eps}\cap\La|\bigr] = \int_\La k_{t,\eps}(x)\,dx
\end{equation}
which should hold for each compact $\Lambda\subset\X$. Henceforth, $\E[\zeta]$ denotes the expected value of a random variable $\zeta$ (with respect to the distribution of~$\ga_{t,\eps}$), and $|\eta|$ denotes number of points in a finite subset $\eta\subset\X$.

It can be shown, see Section~\ref{sec:prelim} below for details, that, for $\eps\to0$,
\begin{equation}\label{eq:meanfielddens}
k_{t,\eps}(x)=q_t(\eps x)+o(1),
\end{equation}
where $q_t(x)\geq 0$ solves the so-called \emph{mean-field}, or \emph{kinetic}, nonlinear equation, see \eqref{eq:fkpp} below. Moreover, 
\[
\lim_{t\to \infty} q_t(\eps x)=0, \quad x\in\X, \ \eps>0,
\]
if and only if $m\geq \kap$, cf. \eqref{eq:nondegen}. It is natural to expect, however, that \eqref{eq:extinction} may take place for smaller value of $m$, because of the term $o(1)$ in \eqref{eq:meanfielddens} which naturally depends on $x\in\X$ and $t\geq0$. To discuss this, one needs the next term of the expansion \eqref{eq:meanfielddens}, using the approach considered in \cite{OFKCBK2014, OC2006, CSFSO2019}. It~yields that, for $\eps\to0$,
\begin{equation}\label{eq:beyondmeanfielddens}
k_{t,\eps}(x)=q_t (\eps x)+\eps^d p_t (\eps x)+o(\eps^d),
\end{equation}
where $p_t(x)$ can be obtained from a coupled system of linear nonhomogeneous and nonautonomous equations \eqref{eq:eqnptx}--\eqref{eq:eqngtx} (see Section~\ref{sec:prelim} for details). 

In the present paper, we consider the space homogeneous regime, when $k_{t,\eps}, q_t, p_t$ do not depend on the space variable. Then both $q_t$ and $p_t$ satisfy ordinary differential equations \eqref{eq:logistic}, \eqref{eq:ptF}, respectively, with 
\[
\lim_{t\to\infty} q_t=\frac{\kap- m}{\kam}=:q^*>0,
\]
for $m<\kap$. The limit $p^*:=\lim\limits_{t\to\infty} p_t$ is found in Theorem~\ref{thm:limits} below. Assuming that that the last term in \eqref{eq:beyondmeanfielddens} also has a limit as $t\to\infty$ of the same order of $\eps$, we get that the extinction, in the space homogeneous case, takes place iff
\begin{equation}\label{eq:exteqnintro}
q^*+\eps^d p^*+o(\eps^d)=0.
\end{equation}

Note that the conditions we imposed typically lead to $p^*<0$, that explains why \eqref{eq:exteqnintro} should take place for $m<\kap$. To formalise this, we replace $m$ by $m(\eps)<\kap$ and reveal the asymptotics of $m(\eps)$ from \eqref{eq:exteqnintro}. We show that (Theorems~\ref{thm:assympd3}, \ref{thm:assympd2}, \ref{thm:assympd1}),
\begin{equation}\label{eq:mainresult}
q^*(\eps):=\frac{\kap-m(\eps)}{\kam}=\begin{cases}
\la_3 \eps^d +o(\eps^d), & d\geq3,\\[2mm]
\la_2 \eps^2 W(\eps^{-2})+ o\bigl(\eps^2 W(\eps^{-2})\bigr), & d=2,\\[2mm]
\la_1 \eps^{\frac{2}{3}}+o\bigl( \eps^{\frac{2}{3}}\bigr), & d=1,
\end{cases}
\end{equation}
where $\la_3,\la_2,\la_1$ are explicit positive constants dependent on $a^+$ and $a^-$. Here $W(x)$ denotes the Lambert W function that solves $W(x)e^{W(x)}=x$ for $x\geq0$; using its known asymptotics we also get that, for the case $d=2$,
\[
\frac{\kap-m(\eps)}{\kam}=-2 \la_2\eps^2\log\eps +o(\eps^2\log\eps).
\]

In other words, we show that the mortality needed to ensure that the population (statistically) will extinct as time tends to infinity is less than $\kap$, namely,
\[
	m(\eps)=\kap-\kam q^*(\eps),
\]  
where $q^*(\eps)>0$ is given by \eqref{eq:mainresult}.

It is worth noting that the orders of the leading terms in the asymptotics \eqref{eq:mainresult} coincide, for all $d\geq1$, with the asymptotics of the critical branching parameter for a lattice contact model considered in \cite{Dur1999,DP1999,BDS1989}, where $\eps$ was the mesh size of the lattice. We expect to discuss a connection between two models as well as to consider the space non-homogeneous case in forthcoming papers. 

The paper is organised as follows. In Section~\ref{sec:prelim}, we describe further details about the spatial and stochastic logistic model, and discuss how one can derive equations on $q_t$ and $p_t$. In~Section~\ref{sec:sphomcase}, we explain the specific of the space-homogeneous case and prove the existence of the limit $p^*=\lim\limits_{t\to\infty}p_t$. In~Section~\ref{sec:critmort}, we introduce $m(\eps)$ and discuss the limits of $q^*(\eps)$ and $p^*(\eps)$ depending on the dimension $d$. Finally, in~Sections~\ref{sec:asympd3}--\ref{sec:asympd1} we find the asymptotics \eqref{eq:mainresult} of $q^*(\eps)$ (and hence of $\mcr(\eps)$) for $d\geq3$, $d=2$, and $d=1$, respectively.

\section{Spatial and stochastic logistic model}\label{sec:prelim}

We consider dynamics of a system consisting of indistinguishable individuals. Each individual is fully characterized by its position $x\in\X$, $d\geq1$. We will always assume that there are not two or more individuals at the same position. 

Let $\Bc(\X)$ denote the set of all Borel subsets of $\X$ with compact closure.
We will consider discrete systems only, finite or locally finite. The latter means that, if $\ga_{t,\eps}=\{x\}$ is a system of individuals at some moment of time $t\geq0$, then we assume that 
$|\ga_{t,\eps}\cap\La|<\infty$ for all $\La\in\Bc(\X)$. In particular, of course, a finite $\ga_{t,\eps}$ is also locally finite. We will call such $\ga_{t,\eps}$ a (finite or locally finite) \emph{configuration}.

The individuals of a configuration are \emph{random}, hence we will speak about random configurations $\ga_{t,\eps}$ with respect to (w.r.t. henceforth) a probability distribution. Let $\Ga$ denote the space of locally finite configurations. We fix the $\sigma$-algebra $\B(\Ga)$ on $\Ga$ generated by all mapppings $\Ga\ni\ga\mapsto |\ga\cap\La|\in\N_0:=\N\cup\{0\}$, $\la\in\Bc(\X)$.

The dynamics of configurations in time $t$ is defined through the dynamics of their distributions. Heuristically, the scheme is as follows. We consider, for an $\eps\in(0,1)$, a mapping on measurable functions $F:\Ga\to\R$ given by
\begin{align}\notag
(L_\eps F)(\ga)&=\sum_{x\in\ga} \biggl( m+\sum_{y\in\ga\setminus\{x\}} a_\eps^-(x-y)\biggr)\Bigl(F \bigl(\ga\setminus\{x\}\bigr) - F\bigl(\ga\bigr)\Bigr)\\
&\quad +\sum_{x\in\ga}\int_\X a_\eps^+(x-y) \Bigl(F \bigl(\ga\cup\{y\}\bigr) - F\bigl(\ga\bigr)\Bigr)dy.\label{eq:genL}
\end{align}
Recall that $m>0$ is a constant and functions $a_\eps^\pm$ are defined through \eqref{eq:aeps}, where $0\leq a^+,a^-\in L^1(\X)$ and \eqref{eq:nondegen} holds.

Operator \eqref{eq:genL} has two properties: 1) $L_\eps 1=0$ and 2)~if, for a given function $F$, a configuration $\ga^* $ is such that $F(\ga^* )\geq F(\ga)$ for all $\ga\in\Ga$ (i.e. if $\ga^* $ is a global maximum for $F$), then $(L_\eps F)(\ga^* )\leq 0$. Hence, formally, $L_\eps $ is a \emph{Markov generator}.

The dynamics of $\ga_{t,\eps}$ if defined then through the differential equation:
\begin{equation}\label{eq:stochdyn}
\dt \E\bigl[F(\ga_{t,\eps})\bigr] = \E\bigl[(L_\eps F)(\ga_{t,\eps})\bigr]
\end{equation}
which should be satisfied for a large class of functions $F$. 

\begin{definition}
A function $k_{t,\eps}:\X\to\R_+:=[0,\infty)$ is said to be \textit{the first order correlation function} (for the distribution of $\ga_{t,\eps}$), if for any function $g(x)\geq0$,
\begin{equation}\label{eq:cfi}
	\E \Bigl[ \sum_{x\in\ga_{t,\eps}} g(x)\Bigr]=\int_\X g(x) k_{t,\eps}(x)\,dx.
\end{equation}
\end{definition}

The function $k_{t,\eps}(x)$ is also called \textit{the density} of individuals of the configuration $\ga_{t,\eps}$, since, taking $g(x)=\1_\La(x)$ for a $\La\in\Bc(\X)$,
we get from \eqref{eq:cfi} that \eqref{eq:kt1def} holds.

\begin{definition}
A symmetric function $k_{t,\eps}^{(2)}:(\X)^2\to\R_+$ is called  \textit{the second-order correlation function}, if, for any symmetric function $g_2:(\X)^2\to\R_+$,
\begin{equation}\label{eq:cfij}
	\E \Bigl[ \sum_{\substack{x \in\ga_{t,\eps}\\ y \in\ga_{t,\eps}\\ x \neq y }} g_2(x ,y )\Bigr]=\int_\X\int_\X g_2(x ,y ) k_{t,\eps}^{(2)}(x ,y )\,dx dy .
\end{equation}
\end{definition}

Combining \eqref{eq:cfij} with \eqref{eq:cfi}, we can also write,
\begin{align}\notag
	\E \Bigl[ \sum_{\substack{x \in \ga_{t,\eps}\\y \in \ga_{t,\eps}}} g_2(x ,y )\Bigr]&=\int_\X\int_\X g_2(x ,y ) k_{t,\eps}^{(2)}(x ,y )\,dx dy \\&\quad +
	\int_\X g_2(x,x) k_{t,\eps}(x)\,dx.\label{eq:cfijgen}
	\end{align}
Substituting to \eqref{eq:cfijgen} the symmetric function
\begin{equation*}\label{eq:specg2}
g_2(x ,y )=\frac{1}{2}\Bigl(\1_{\La_1}(x )\1_{\La_2}(y )+\1_{\La_1}(y )\1_{\La_2}(x )\Bigr),
\end{equation*}
where $\La_1,\La_2\in\Bc(\X)$, we get
\begin{align*}
\E \bigl[|\ga_{t,\eps}\cap\La_1|\,
|\ga_{t,\eps}\cap\La_2|\bigr] &= \int_{\La_1} \int_{\La_2} k_{t,\eps}^{(2)}(x ,y )\,dx dy  +
\int_{\La_1\cap\La_2} k_{t,\eps} (x)\,dx,
\end{align*}
and hence the \emph{covariance} between random numbers $|\ga_{t,\eps}\cap\La_1|$ and $|\ga_{t,\eps}\cap\La_2|$ is given by
\begin{align}
&\E \Bigl[\Bigl(|\ga_{t,\eps}\cap\La_1|-\E\bigl[|\ga_{t,\eps}\cap\La_1|\bigr]\Bigr)\,
\Bigl(|\ga_{t,\eps}\cap\La_2|-\E\bigl[|\ga_{t,\eps}\cap\La_2|\bigr]\Bigr)\Bigr] \label{eq:covar}\\
&= \E \bigl[|\ga_{t,\eps}\cap\La_1|\,
|\ga_{t,\eps}\cap\La_2|\bigr]-\E \bigl[|\ga_{t,\eps}\cap\La_1|\bigr]\,\E \bigl[
|\ga_{t,\eps} \cap\La_2|\bigr]\notag\\
&= \int_{\La_1} \int_{\La_2} \Bigl(k_{t,\eps}^{(2)} (x ,y )-k_{t,\eps} (x )\,k_{t,\eps} (y )\Bigr)\,dx dy   + 
\int_{\La_1\cap\La_2} k_{t,\eps}  (x)\,dx.\label{eq:covarans}
\end{align}

Substituting \eqref{eq:genL} into \eqref{eq:stochdyn} and using \eqref{eq:cfi} and \eqref{eq:cfij}, we obtain that $k_{t,\eps}(x)$ satisfies the following equation 
\begin{align*}
 \pdt k_{t,\eps}(x)&=\int_\X a_\eps^+(x-y)k_{t,\eps}(y)\,dy - m k_{t,\eps}(x)\notag\\
&\quad - \int_\X a_\eps^-(x-y)k_{t,\eps}^{(2)}(x,y)\,dy,
 \end{align*}
see  e.g. \cite{FKK2009,FKK2011a} for details. Similarly, the evolution of $k_{t,\eps}^{(2)}(x,y)$ depends on the third order correlation function and so on. 

It can be shown, see \cite{FM2004,FKK2011a,FKKozK2014}, that then, for $\eps\to0$,
\begin{equation}\label{eq:meanfield}
\begin{aligned}
k_{t,\eps}(x)&=q_t(\eps x)+o(1), \\
k_{t,\eps}^{(2)}(x ,y )&=q_t(\eps x )q_t(\eps y )+o(1),\end{aligned}
\end{equation}
where $q_t $ solves the following \emph{mean-field}, or \emph{kinetic}, equation 
\begin{align}
\pdt q_t (x) &=\int_\X a^+(x-y)q_t(y)\,dy - m q_t(x)\notag\\
&\quad - q_t(x) \int_\X a^-(x-y)q_t(y)\,dy.\label{eq:fkpp}
\end{align}
Note that it was shown using another scaling, which apparently is equivalent to the considered one, see \cite{OFKCBK2014} for details. For various properties of solutions to \eqref{eq:fkpp}, see \cite{FT2017c,FKT100-1,FKT100-2,FKT100-3,FKMT2017,KT2019,FKT2019md}.

The asymptotics \eqref{eq:meanfield} however does not describe effectively the covariance \eqref{eq:covar} between 
random numbers $|\ga_{t,\eps}\cap\La_1|$ and $|\ga_{t,\eps}\cap\La_2|$, especially in the case of disjoint $\La_1,\La_2\in\Bc(\X)$, since then, by \eqref{eq:covarans} and \eqref{eq:meanfield},
\[
\E \Bigl[\Bigl(|\ga_{t,\eps}\cap\La_1|-\E\bigl[|\ga_{t,\eps}\cap\La_1|\bigr]\Bigr)\,
\Bigl(|\ga_{t,\eps}\cap\La_2|-\E\bigl[|\ga_{t,\eps}\cap\La_2|\bigr]\Bigr)\Bigr] = o(1).
\]

To partially reveal the covariance above, one needs hence an enhanced asymptotics \eqref{eq:meanfield}. A mathematical approach for this was proposed in \cite{OFKCBK2014}, justifying the heuristic considerations in the early publication \cite{OC2006}; the approach has been recently generalised in \cite{CSFSO2019}. Namely, it was shown that
\begin{equation}\label{eq:beyondmeanfield}
\begin{aligned}
k_{t,\eps}(x)&=q_t (\eps x)+\eps^d p_t (\eps x)+o(\eps^d),\\
k_{t,\eps}^{(2)} (x ,y )&=q_t (\eps x )q_t (\eps y )+\eps^d g_t  (\eps x ,\eps y )+o(\eps^d),
\end{aligned}
\end{equation}
where
\begin{align}
\pdt  p_t(x) &=\int_\X   a^+    (x- y)  p_t(y) \,dy-m  p_t(x) - q_t(x)\int_\X   a^-   (x-y)  p_t(y) \,dy\notag \\& \ \ -  p_t(x) \int_\X  a^-   (x-y)  q_t(y) \,dy -\int_\X g_t(x, y) a^-   (x-y)\,dy; \label{eq:eqnptx}
\end{align}
and
\begin{align}
\pdt g_t(x,y)&=
\int_\X   [ g_t(x,z) a^+    (y- z)+ g_t(z,y) a^+    (x-z)]\,dz \notag -2mg_t(x,y) \\&\ \ 
-g_t(x,y)\int_\X  [ a^-   (x-z)+ a^-   (y-z)] q_t(z)\,dz
\notag \\ &\ \   + a^+    (x-y)[ q_t(x) + q_t(y) ] -2 a^-   (y-x) q_t(y)  q_t(x) \notag \\
&\ \  - \int_\X   [ a^-   (x-z) q_t(x) g_t(z,y)+ a^-   (y-z) q_t(y) g_t(x, z)]\,dz .
\label{eq:eqngtx} 
\end{align}

\section{Space-homogeneous case}\label{sec:sphomcase}

Let henceforth, for an integrable function $f$ on $\X$, $\hat{f}$ denote its \emph{unitary} Fourier transform given by
\begin{equation}\label{eq:fourier}
\hat{f}(\xi):=\int_\X f(x) e^{- 2 i \pi x\cdot \xi}\,d\xi,
\end{equation}
where $x\cdot \xi$ denotes the standard dot-product in $\X$. 
Note that
\begin{equation}\label{eq:fourest}
|\hat{f}(\xi)|\leq \int_\X |f(x)|\,dx, \qquad \xi\in\X. 
\end{equation}

We formulate now our basic assumptions on the kernels $a^\pm:\X\to [0,\infty)$:
\begin{equation}\tag{\textbf{A1}}\label{eq:newA1}
\begin{gathered}
a^\pm\in L^1(\X)\cap L^\infty(\X); \qquad \hat{a}^\pm\in L^1(\X)\\
a^\pm(-x)=a^\pm(x),\quad x\in\X.
\end{gathered}
\end{equation}
Note that \eqref{eq:newA1}, together with \eqref{eq:fourest}, imply that $a^\pm,\hat{a}^\pm\in L^2(\X)$.
It is also well-known that $\hat{a}^\pm$ are (uniformly) continuous functions on $\X$.

Equation \eqref{eq:fkpp} has two constant stationary solutions $q_t(x)=0$ and $q_t(x)=q^*$, where
\begin{equation}\label{eq:theta}
q^*:=\frac{\kap -m}{\kam }.
\end{equation}
We will always assume that
\begin{equation}\label{eq:nondeg}
\kap >m, \tag{\textbf{A2}}
\end{equation}
i.e. that $ q^*  >0$; otherwise, the solution to \eqref{eq:fkpp} with $q_0(x)\geq0$, $x\in\X$, would uniformly degenerate as $t\to\infty$. We assume also that 
\begin{equation}\label{eq:compar}
J^* (x):= a^+(x)- q^*   a^-(x)\geq 0, \qquad x\in\X.\tag{\textbf{A3}}
\end{equation}
The reason for this restriction is as follows. By \eqref{eq:fourest},
assumption \eqref{eq:compar} yields
\begin{equation}
|\hat{J}^* (\xi)|\leq \int_\X |J^* (x)|\,dx
	=\int_\X J^* (x)\,dx =\kap  - q^* \kam =m. \label{eq:lessm}
\end{equation}
Next $a^\pm(-x)=a^\pm(x)$ for $x\in\X$ implies that $\hat{a}^\pm(\xi)\in\R$ for $\xi\in\X$, and therefore,
\begin{equation}\label{eq:posFT}
	\kap -\hat{J}^* (\xi)\geq \kap -m>0, \qquad \xi\in\X.
\end{equation}
If \eqref{eq:posFT} fails, then (under further assumptions) there exists an infinite family of non-constant (in space) stationary solutions to \eqref{eq:fkpp}, see \cite{KT2019}.

Under assumption \eqref{eq:compar}, if $q_0(x)=q_0$ for all $x\in\X$, then, by \cite[Proposition~2.7]{FKT100-1}, the solution to \eqref{eq:fkpp} is also space homogeneous: $q_t(x)=q_t$, where $q_t$ solves the logistic differential equation
\begin{equation}\label{eq:logistic}
\dt q_t= \kap  q_t - mq_t - \kam  q_t^2=\kam q_t(q^*-q_t).
\end{equation}
It is straightforward to check that then
\begin{equation}\label{eq:logisticsol}
q_t =\frac{ q^*  q_0}{q_0+( q^*  -q_0) e^{-(\kap -m) t}},
\end{equation}
hence 
\begin{equation}\label{eq:limqt}
\lim_{t\to\infty} q_t=q^*.
\end{equation}

We will assume henceforth that
\begin{equation}\label{eq:q0small}
0<q_0< q^* ,
\end{equation}
then, by \eqref{eq:logisticsol},
\begin{equation} \label{eq:qtbelow}
0<q_t< q^*, \qquad t>0.
\end{equation}
Note that then, by \eqref{eq:logisticsol},  $\dt q_t > 0$ for $t>0$, i.e. $q_t$ is (strictly) increasing.

Equations \eqref{eq:eqnptx} and \eqref{eq:eqngtx} are linear, and it is straightforward to check that, in the space-homogeneous case, when $p_0(x)= p_0$, $g_0(x,y)=g_0(x-y)$ for all $x,y\in\X$, this property will be preserved in time, so that \eqref{eq:beyondmeanfield} takes the form
\begin{align}
k_{t,\eps} (x)&=k_{t,\eps} =q_t +\eps^d p_t  +o(\eps^d),\label{eq:beyondmfhom1} \\
k_{t,\eps} (x,y)& = k_{t,\eps} (x-y)= q_t  q_t  + \eps^d g_t (x-y)+o(\eps^d),\notag
\end{align}
where, recall $q_t$ solves \eqref{eq:logistic} and hence is given by \eqref{eq:logisticsol}, and the equations for $p_t,g_t(x)$ take the following form:
\begin{align}
\dt  p_t &= \kap        p_t -m p_t  - 2 \kam     q_t  p_t  -\int_\X g_t(y) a^-   (y)\,dy;\label{eq:ptfirst}\\
\pdt g_t(x)&=2\int_\X    a^+    (x- y) g_t(y)\,dy -2 \kam  q_t  g_t(x) -2mg_t(x)\notag\\&\quad   +2 a^+    (x) q_t-2 a^-   (x) q_t ^2-2 q_t  \int_\X    a^-   (x- y) g_t(y)\,dy.\label{eq:gtfirst}
\end{align}

For $q_0\in(0,q^*)$, \eqref{eq:qtbelow} holds, and hence
\begin{equation}\label{eq:defJt}
J_t(x):=a^+(x)-q_t a^-(x)> J^* (x)\geq0.
\end{equation}
One can rewrite then \eqref{eq:gtfirst}:
\begin{equation}\label{eq:gt}
\pdt g_t(x)=2\int_\X    J_t    (x- y)  g_t(y)\,dy
 -2 (\kam  q_t+m)  g_t(x)  
  +2  q_t J_t (x) .
\end{equation}

It is straightforward to check that if $g_0\in L^1(\X)\cap L^\infty(\X)$, then $g_t\in L^1(\X)\cap L^\infty(\X)$ for all $t>0$. One can apply then the Fourier transform to both parts of \eqref{eq:gt} to get
\begin{equation}\label{eq:gtF}
\pdt \hat{g}_t(\xi) = 2 \bigl(\hat{J}_t(\xi)-\kam  q_t-m\bigr)\hat{g}_t(\xi)+ 2  q_t \hat{J}_t (\xi).
\end{equation}
By the above, $g_t,a^-\in L^2(\X)$, $t\geq0$, and since we have chosen the unitary Fourier transform \eqref{eq:fourier}, we can rewrite \eqref{eq:ptfirst}, by using the Parseval identity, as follows:
\begin{equation}\label{eq:ptF}
\dt  p_t = \kap        p_t -m p_t  - 2 \kam     q_t  p_t  -\int_\X \hat{g}_t(\xi) \hat{a}^-(\xi)\,d\xi.
\end{equation}

We are going to find limits of $\hat g_t$ and $p_t$ as $t\to\infty$. To this end, we prove an abstract lemma which is actually an adaptation of e.g.~\cite[Theorem~5.8.2]{Paz1983} to the case of bounded operators (that apparently requires wicker conditions).

\begin{lemma}\label{le:tobeused}
Let $\bigl(X,\lVert\cdot\rVert_X\bigr)$ be a Banach space, and let $\bigl(\mathcal{L}(X),\lVert\cdot\rVert
\bigr)$ denote the Banach space of linear bounded operators on $X$. Let $A\in C\bigl([0,\infty)\to\mathcal{L}(X)\bigr)$ be a continuous operator-valued function. Suppose that there exists $c,\nu>0$ such that, for all $t\geq s\geq0$, the operator
\[
	U(t,s):=\exp\biggl( \int_s^t A(\tau)\,d\tau \biggr)\in\mathcal{L}(X)
\]
satisfies
\begin{equation}\label{eq:Uexp}
\bigl\lVert U(t,s)\bigr\rVert \leq c e^{-\nu(t-s)}.
\end{equation}
Let $f\in C\bigl([0,\infty),X\bigr)$ be a continuous $X$-valued function. Suppose that $f(t)$ converges in $X$ to some $f(\infty)\in X$ and $A(t)$ strongly converges to some $A(\infty)\in \mathcal{L}(X)$ as $t\to\infty$. 
Finally, suppose that $A(\infty)$ is an invertible operator, i.e. that there exists $A(\infty)^{-1}\in\mathcal{L}(X)$. Then the unique classical solution to the following non-homogeneous Cauchy problem in $X$:
\begin{equation}\label{eq:Cauchy}
\dfrac{d}{dt} u(t) = A(t)u(t)+f(t), \qquad u(0)=u_0\in\X,
\end{equation}
converges in $X$, as $t\to\infty$, to
\[
	-A(\infty)^{-1}f(\infty)\in X.
\]
\end{lemma}
\begin{proof}
Since $A\in C_b\bigl([0,\infty)\to\mathcal{X}\bigr)$, the unique classical solution $u\in C_b\bigl([0,\infty),X\bigr)$ to \eqref{eq:Cauchy} (i.e. such that $u\in C^1\bigl((0,\infty),X\bigr)$) is given by
\begin{equation}\label{eq:wfewrew}
	u(t)=U(t,0)u(0)+\int_0^t U(t,s) f(s)\,ds,
\end{equation}
see e.g.~\cite[Chapter~3]{DK1974} (all integrals are in the sense of Bochner henceforth). By~\eqref{eq:Uexp},
\begin{equation}\label{eq:3r325532}
\|U(t,0)u(0)\|_X\leq e^{-\nu t}\|u(0)\|_X\to0,\qquad t\to\infty.
\end{equation}

Suppose, firstly, that $f(\infty)=0$. Since then $f(s)\to0$ in $X$ as $s\to\infty$, one gets that, for any $\eps>0$, there exists $T=T(\eps)>0$ such that $\|f(s)\|\leq \eps$ for all $s\geq T$. 
Since $f\in C\bigl([0,\infty),X\bigr)$, one can define
$\|f\|_T:=\sup\limits_{t\in[0,T]}\|f(t)\|_{X}<\infty$. Then
\begin{align*}
 \biggl\lvert \int_0^t U(t,s) f(s)\,ds\biggr\rVert&\leq
\int_0^T \lVert U(t,s)\rVert \|f(s)\|_X\,ds
+\eps \int_T^t \lVert U(t,s)\rVert \,ds\\
&\leq c\|f\|_T \int_0^T e^{-\nu(t-s)}\,ds+
c\,\eps \int_T^t e^{-\nu(t-s)}\,ds\\
&\leq c\|f\|_T\frac{1}{\nu}e^{-\nu(t-T)}+\frac{c \,\eps}{\nu},
 \end{align*}
and combining this with \eqref{eq:3r325532}, one gets that $u(t)\to 0=u(\infty)$ in $X$ as $t\to\infty$. 

For a general $f(\infty)\in X$, denote $v(t):=u(t)-u(\infty)$. Since $A(\infty)$ is invertible, one can define
\[
	u(\infty):=-A(\infty)^{-1} f(\infty)\in X.
\]
Then
\[
	\dfrac{d}{dt} v(t)=\dfrac{d}{dt}u(t)=A(t)u(t)+f(t)
	=A(t)v(t)+f(t)+A(t)u(\infty).
\]
We set $g(t):=f(t)+A(t)u(\infty)$, $t\geq0$. 
By the assumptions on $f$ and $A$, $g\in C\bigl([0,\infty),X\bigr)$ and 
\[
\lim_{t\to\infty} g(t)=f(\infty)+A(\infty)\bigl( -A(\infty)^{-1}f(\infty) \bigr) =0,
\]  
where the limit is in $X$. Then, by the proved above, $v(t)\to 0$ in $X$, and hence $u(t)\to u(\infty)$ in $X$.
\end{proof}

\begin{theorem}\label{thm:limits}
Let \eqref{eq:newA1}--\eqref{eq:compar} hold. Let $q_0$ satisfies \eqref{eq:q0small} and $g_0, \hat{g}_0\in L^1(\X)\cap L^\infty(\X)$. Then there exist limits
\begin{align}
\hat{g}^* (\xi):&=\lim_{t\to\infty} \hat{g}_t(\xi)=\frac{q^*  \hat{J}^* (\xi)}{\kap -\hat{J}^* (\xi)}\leq \frac{m}{\kam}, \qquad \xi\in\X, \label{eq:gfstat}\\
p^*:&=\lim_{t\to\infty} p_t=-\frac{1}{\kam }\int_\X \frac{\hat{J}^* (\xi)}{\kap -\hat{J}^* (\xi)}\hat{a}^-(\xi)\,d\xi\in\R.
\label{eq:pstat}
\end{align} 
Moreover, the convergence in \eqref{eq:gfstat} takes place in the norms of both $L^1(\X)$ and $L^\infty(\X)$.
As a result, $g_t$ converges, as $t\to\infty$,  in $L^\infty(\X)$ to $g^*$, the inverse Fourier transform of $\hat{g}^*$.
\end{theorem}
\begin{remark}
We will actually use in the proof a part of the estimate \eqref{eq:lessm} only, rather than the more strict assumption \eqref{eq:compar}. More precisely, it is easy to check that all arguments of the proof remain correct if the assume, instead of \eqref{eq:compar}, that, for some $\alpha\in(0,\kap)$,
\[
	\kap-\hat{J}^* (\xi)\geq \alpha, \qquad \xi\in\X;
\]
the uppear bound in \eqref{eq:gfstat} will be then replaced by $q^*\frac{\kap- \alpha}{\alpha}$.
\end{remark}
\begin{proof}[Proof of Theorem~\ref{thm:limits}]
We denote
\begin{equation}\label{eq:defj}
j(\xi,t):=\hat{J}_t(\xi)-\kam  q_t-m, \qquad \xi\in\X,\ t\geq0,
\end{equation}
and apply Lemma~\ref{le:tobeused} to equation \eqref{eq:gtF}, where $X=L^1(\X)$ or $X=L^\infty(\X)$, $A(t)$ is the multiplication operator by the function $2j(\xi,t)$, and $f(t,\xi)=2q_t\hat{J}_t(\xi)$. Note that, for any $t\geq s\geq0$, $\xi\in\X$, we have, by \eqref{eq:newA1}, \eqref{eq:fourest}, \eqref{eq:qtbelow},
\begin{align}\notag
\hat{J}_t(\cdot), j(\cdot,t)&\in L^1(\X)\cap L^\infty(\X),\\
|j(\xi,t)-j(\xi,s)|&\leq 2|\hat{J}_t(\xi)-\hat{J}_s(\xi)|+2\kam|q_t-q_s| \leq 4\kam \lvert q_t-q_s\rvert,\notag\\
|f(t,\xi)-f(s,\xi)|&\leq 2|q_t-q_s||\hat{J}_t(\xi)|
+2q_s|\hat{J}_t(\xi)-\hat{J}_s(\xi)|\notag\\
&\leq 2\bigl(|\hat{a}^+(\xi)|+q^*(q^*+1)|\hat{a}^-(\xi)| \bigr)|q_t-q_s|.\label{eq:AFR3WRT}
\end{align}
Therefore, $A\in C\bigl([0,\infty)\to\mathcal{L}(X)\bigr)$ and $f\in C\bigl([0,\infty),X\bigr)$ for both $X=L^1(\X)$ and $X=L^\infty(\X)$. Note that, by \eqref{eq:wfewrew},
\[
	\hat{g}_t(\xi) = \exp\biggl( 2 \int_0^t j(\xi,\tau)\,d \tau\biggr)\hat{g}_0(\xi)
		+\int_0^t\exp\biggl( 2 \int_s^t j(\xi,\tau)\,d\tau\biggr)  q_s \hat{J}_s (\xi)\,ds,
\]
and then, by the Riemann--Lebesgue lemma, $\hat{g}_t(\cdot)\in C(\X)$, $t\geq0$. 

By \eqref{eq:AFR3WRT}, we also have that $f(t,\xi)$ converges, in the norm of either of $X$ to $2q^*\hat{J}^*(\xi)$; and also $A(t)$ strongly converges to the operator $A(\infty)$ of the multiplication by
\[
	2\lim_{t\to\infty} \hat{J}_t(\xi)=2(\hat{J}^*(\xi)-\kam  q^*-m)=2(\hat{J}^*(\xi)-\kap).
\] 
By \eqref{eq:posFT}, operator $A(\infty)$ is invertible.

Next, for all $t> s\geq0$, $\int_s^t A(\tau)\,d \tau$ is the operator of multiplication by $2\int_s^t j(\cdot,\tau)\, d \tau$. We have
\begin{equation}
\int_s^t j(\xi,\tau)\,d \tau=\bigl(\hat{a}^+(\xi)-m\bigr)(t-s)-(\hat{a}^-(\xi)+\kam )\int_s^t  q_\tau\,d \tau.\label{eq:intermed}
\end{equation}
Since, by \eqref{eq:logistic},
\begin{equation*}
	\dt \log  q_t =\frac{1}{ q_t }\dt q_t =\kap -m -  \kam    q_t,
\end{equation*}
we get
\[
	\log q_t-\log q_s  = \int_s^t \frac{d}{d \tau}\log  q_\tau d \tau= (\kap -m )(t-s)-  \kam  \int_s^t   q_\tau d \tau,
\]
and hence
\begin{equation} \label{eq:intofqt}
\int_s^t  q_\tau d \tau = q^* (t-s)-\frac{1}{\kam }\log \frac{q_t}{q_s}.
\end{equation} 
Substituting \eqref{eq:intofqt} into \eqref{eq:intermed}, and using \eqref{eq:posFT} and that $q_t$ is increasing and $|\hat{a}^-(\xi)|\leq \kam $ holds, we get
\begin{equation}
\int_s^t j(\xi,\tau)\,d \tau= -\bigl(\kap -\hat{J}^* (\xi)\bigr) (t-s)+\frac{\hat{a}^-(\xi)+\kam }{\kam }\log \frac{q_t}{q_s}.\label{eq:return}
\end{equation}

Therefore, cosnidering a multiplication operator $U(t,s)=\exp\Bigl(\int_s^t A(\tau)\,d \tau \Bigr)$ and using \eqref{eq:posFT} and that
\[
	|\hat{a}^-(\xi)| \leq \kam , \qquad 0<q_0\leq q_s<q_t<q^*, \quad t>s>0,
\]
we get from \eqref{eq:return} that, in either of spaces $X$,
\begin{equation}
\|U(t,s)\|=\sup_{\xi\in\X}\exp\biggl(2\int_s^t j(\xi,\tau)\,d \tau\biggr)\leq \biggl(\frac{q^*}{q_0}\biggr)^4 e^{-2(\kap -m)(t-s)}. \label{eq:intest0}
\end{equation}

Therefore, by Lemma~\ref{le:tobeused}, 
\[
	\hat{g}_t(\xi)\to - A(\infty)^{-1}f(\infty)
	=-\frac{1}{2(\hat{J}^*(\xi)-\kap)}2q^*\hat{J}^*(\xi)=
	\frac{q^*\hat{J}^*(\xi)}{\kap-\hat{J}^*(\xi))}=:\hat{g}^*(\xi)
\]
in the sense of norm in both $L^1(\X)$ and $L^\infty(\X)$ (and, in particular, pointwise). We also have, by  \eqref{eq:newA1}, that
\[
	|\hat{g}^*(\xi)|\leq q^*\frac{m}{\kap -m}=\frac{m}{\kam }, \qquad \xi\in\X,
\]
that finishes the proof of \eqref{eq:gfstat}. 

Since $\hat{g}_t\in L^1(\X)$, its inverse Fourier transform coinsides with $g_t$ a.e.; in~particular, they coincide as elements of $L^\infty(\X)$. Hence if $g^*$ is the inverse Fourier transform of $\hat{g}^*\in L^1(\X)$, then, by \eqref{eq:fourest}, 
\[
	\|g_t-g^*\|_{L^\infty(\X)}\leq \|\hat g_t-\hat g^*\|_{L^1(\X)}\to0,\qquad t\to\infty,
\]
that proves the last statement of Theorem~\ref{thm:limits}.

We are going to apply now Lemma~\ref{le:tobeused} to equation \eqref{eq:ptF}, with $X=\R$. 
Since $\hat{g}_t\to \hat{g}^*$ in $L^\infty(\X)$ and $\hat{g}_t$ is a classical solution to \eqref{eq:gtF} in $L^\infty(\X)$ (i.e. a continuous mapping from $[0,\infty)$ to $L^\infty(\X)$), function $\hat{g}_t(\xi)$ is globally bounded in $t\geq0$ and $\xi\in\X$. Then
\[
	-\int_\X \hat{g}_t(\xi) \hat{a}^-(\xi)\,d\xi\to-\int_\X \hat{g}^*(\xi) \hat{a}^-(\xi)\,d\xi, \qquad t\to\infty,
\]
by the dominated convergence theorem. $A(t)$ is given now through the multiplication by  $c_t:=\kap -m   - 2 \kam  q_t$, and, by \eqref{eq:intofqt},
\begin{align*}
\int_s^t c_\tau d \tau&=(\kap -m)(t-s)- 2 \kam  q^* (t-s)+2 \log \frac{q_t}{q_s}\notag\\
&=-(\kap -m)(t-s)+2 \log \frac{q_t}{q_s}\leq-(\kap -m)(t-s)+2 \log \frac{q^*}{q_0}.
\end{align*}
Hence, by the same arguments as before, we can apply Lemma~\ref{le:tobeused}: since, by \eqref{eq:gfstat},
\[
	\lim_{t\to\infty} c_t=\kap -m   - 2 \kam  q^*=
	-\kam q^*,
\]
we get
\[
	p_t\to -\frac{1}{\kam q^*}\int_\X \hat{g}^*(\xi) \hat{a}^-(\xi)\,d\xi,
\]
that implies \eqref{eq:pstat}.
\end{proof}

\section{Critical mortality}\label{sec:critmort}

We are going to discuss now the extinction regime.
Recall that $o(\eps^d)$ in \eqref{eq:beyondmfhom1} depends on $t$, so we have
\[
k_{t,\eps}=q_t+\eps^d p_t+o_t(\eps^d),
\]
where, for each $t>0$, 
\[
\lim_{\eps\to0}\frac{o_t(\eps^d)}{\eps^d}=0.
\]
We will assume that
\begin{equation*}
\lim_{t\to\infty}o_t(\eps^d)=o(\eps^d).
\end{equation*}
Then, the extinction \eqref{eq:extinction} takes place if and only if \eqref{eq:exteqnintro} holds.

We fix an $m\in(0,\kap )$ for which \eqref{eq:compar} holds. 
We consider a function $\mcr:(0,1)\to (m,\kap)$ and set, cf. \eqref{eq:theta}, for $\eps\in(0,1)$,
\begin{align}\label{eq:qepsbdd}
q^*(\eps):&=\frac{\kap -\mcr (\eps)}{\kam }\in\Bigl( 0,\frac{\kap -m}{\kam }\Bigr),\\
\intertext{and also, cf. \eqref{eq:compar},} \label{eq:defjeps}
J_\eps(x):&=a^+ (x)- q^*(\eps)  a^- (x)\geq0,
\\
\intertext{because of \eqref{eq:compar}, \eqref{eq:qepsbdd}. Finally, we set, cf. \eqref{eq:pstat}, for $\eps\in(0,1)$,}
\label{eq:pepsint}
p^*(\eps):&=-\frac{1}{\kam }\int_\X \frac{\hat{J}_\eps (\xi)}{\kap -\hat{J}_\eps (\xi)}\hat{a}^-(\xi)\,d\xi.
\end{align}
Note that, by \eqref{eq:fourest}, \eqref{eq:defjeps}, 
\begin{equation}\label{eq:ineqhatJeps}
|\hat{J}_\eps (\xi)|\leq \int_\X 
\bigl\lvert {J}_\eps (x)\bigr\rvert\,dx=\int_\X {J}_\eps(x) \,dx = \mcr (\eps),
\end{equation}
and hence $p^*(\eps)$ is well-defined: \eqref{eq:ineqhatJeps} and \eqref{eq:newA1} yield
\begin{equation*}
|p^*(\eps)|\leq \frac{1}{\kam }\frac{\mcr (\eps)}{\kap -\mcr (\eps)}\int_\X \bigl\lvert \hat{a}^-(\xi)\bigr\rvert\,d\xi<\infty, \qquad \eps\in(0,1).
\end{equation*} 

In the rest of the paper, our main object of interest will be the following equation, cf. \eqref{eq:exteqnintro},
\begin{equation}\label{eq:mcrchar}
q^*(\eps)+ \eps^d p^*(\eps)+o(\eps^d)=0.
\end{equation}

\begin{proposition}\label{prop:evidentlimit}
Let \eqref{eq:newA1}--\eqref{eq:compar} hold. If  \eqref{eq:mcrchar} holds and there exists $\lim\limits_{\eps\to0}\mcr (\eps)$, then
\begin{equation}\label{eq:mcritconv}
\lim_{\eps\to0}\mcr (\eps)=\kap, \qquad 
\lim_{\eps\to0}q^*(\eps)=0.
\end{equation}
\end{proposition}
\begin{proof}
Clearly, $\mcr (0):=\lim\limits_{\eps\to0}\mcr (\eps)\leq \kap$. Suppose that $\mcr (0)<\kap$. Let $\alpha\in(0,1)$ be such that $\alpha \kap>\mcr (0)$. Then there exists $\eps_\alpha\in(0,1)$ such that $\mcr (\eps)<\alpha \kap$ for all $0<\eps<\eps_\alpha$. Therefore, by \eqref{eq:ineqhatJeps}, 
\[
|\hat{J}_\eps (\xi)| <\alpha \kap, \qquad \xi\in\X, \ 0<\eps<\eps_\alpha.
\]
Then, by \eqref{eq:pepsint}, \eqref{eq:newA1},
\[
	|p^*(\eps)|\leq \frac{1}{\kam }\int_\X \frac{\alpha \kap}{\kap -\alpha \kap}\hat{a}^-(\xi)\,d\xi<\infty,
\]
and hence $\eps^d p^*(\eps)\to0$, $\eps\to0$. Therefore, by \eqref{eq:mcrchar}, we get that $q^*(\eps)\to0$ and hence, by \eqref{eq:qepsbdd}, $\mcr (\eps)\to\kap$ that contradicts the assumption. The statement is proved.
\end{proof}

The behaviour of $p^*(\eps)$ as $\eps\to0$ depends on the dimension $d\in\N$: the limit (as $\eps\to0$) of the integrand in \eqref{eq:pepsint} is equal to, because of \eqref{eq:mcritconv}, 
$\dfrac{\hat{a}^+(\xi)\hat{a}^-(\xi)}{\kap-\hat{a}^+(\xi)}$
that has a singularity at the origin, which is, in general, non-integrable for $d<3$. We discuss this under the following additional assumption:
\begin{equation}\tag{\textbf{A4}}\label{eq:secmom}
\int_\X |x|^2a^+(x)\,dx<\infty.
\end{equation}
Since $a^+ \in L^\infty(\X)$, the inequality in \eqref{eq:secmom} implies that  $\int_\X |x|a^+(x)\,dx<\infty$, and using that $a^+(-x)=a^+(x)$, $x\in\X$, we get
\[
\int_\X x a^+ (x)\,dx=0\in\X.
\]
Then, by the Taylor expansion for $\hat{a}^+(\xi)$ defined by \eqref{eq:fourier}, we get (cf. \cite[Corollary 1.2.7]{Sas2013} for another coefficients of the Fourier transform) that
\[
	\hat{a}^+(\xi)=\hat{a}^+(0)-2\pi^2\sum_{i,j=1}^d a^+_{i,j} \xi_i\xi_j+ o(|\xi|^2)=\kap -2\pi^2 A^+\xi\cdot\xi+ o(|\xi|^2),
\]
for $\xi\to0$, where
\begin{equation}\label{eq:defofaplus}
a^+_{i,j}  := \int_\X x_ix_j a^+ (x)\,dx, \qquad 1\leq i,j\leq d,
\end{equation}
and hence $A^+:=\bigl( a_{i,j} \bigr)_{i,j=1}^d$ is a Hermitian (strictly) positive definite matrix. Then there exists a Hermitian (strictly) positive definite matrix $B^+$ such that $(B^+)^2=A^+$, and hence
\begin{equation}\label{eq:Texpcf}
\kap - \hat{a}^+(\xi)=2\pi^2|B^+ \xi|^2 + o(|\xi|^2), \qquad \xi\to0.
\end{equation}

Under assumptions \eqref{eq:newA1}--\eqref{eq:compar}, assumption \eqref{eq:secmom} holds with $a^+$ replaced by $a^-$ and hence by $J_\eps$ or $J^*$. Let $B^-,B_\eps,B^*$ be the Hermitian positive definite matrices corresponding to the functions $a^-,J_\eps,J^*$, respectively. Then, the corresponding analogues of \eqref{eq:Texpcf} hold, with, in particular, $\kap$ replaced by $\kam=\hat{a}^\pm(0)$, $\mcr(\eps)=\hat{J}_\eps(0)$, $m=\hat{J}^*(0)$, respectively. It is easy to see also that
\begin{equation}\label{eq:soineedthis}
|B_\eps \xi|^2=|B^+\xi|^2-q^*(\eps)|B^-\xi|^2, \qquad \xi\in\X, \ \eps\in(0,1).
\end{equation}

Next, for any invertible matrix $B$,
\begin{equation}\label{eq:dblest}
\bigl( \bigl\lVert (B )^{-1}\bigr\rVert \bigr)^{-1}|\xi| \leq |B  \xi| \leq \|B \||\xi|, \qquad \xi\in\X,
\end{equation} 
Then, for small enough $\delta>0$,
\begin{equation}\label{eq:dblestlemma}
\frac{\pi^2}{\lVert (B^\pm)^{-1} \rVert^2}\lvert \xi\rvert ^2 \leq \kapm-\hat{a}^\pm(\xi)\leq 3\pi^2\lVert B^\pm\rVert^2 \lvert \xi\rvert ^2, \qquad |\xi|\leq \delta.
\end{equation}
The corresponding double inequalities can be also obtained for $J_\eps$ and $J^*$.

\begin{proposition}\label{eq:limitofpstar}
Let \eqref{eq:newA1}--\eqref{eq:secmom} hold. 
Let also $\mcr:(0,1)\to (m,\kap)$ and $p^*(\eps)$, defined through \eqref{eq:qepsbdd}--\eqref{eq:pepsint}, be such that \eqref{eq:mcritconv} holds. 
Then, for $d\geq3$,
\begin{equation}\label{eq:pstar0}
\lim_{\eps\to0}p^*(\eps)=-\frac{1}{\kam }\int_\X \frac{\hat{a}^+(\xi)\hat{a}^-(\xi)}{\kap -\hat{a}^+(\xi)}\,d\xi=:p^*(0)\in\R,
\end{equation}
whereas, for $d\leq 2$, $\lim\limits_{\eps\to0}p^*(\eps)=-\infty$.  
\end{proposition}
\begin{remark}
Note that we do not need to assume \eqref{eq:mcrchar} to get the statement. 
\end{remark}
\begin{proof}
For each $\delta>0$, one can expand $p^*(\eps)$ as follows
\begin{align*}
p^*(\eps)&=p_{\leq \delta}^*(\eps)+p_{\geq \delta}^*(\eps)\\ &=:
-\frac{1}{\kam }\int_{|\xi|\leq \delta}\frac{\hat{J}_\eps (\xi)}{\kap -\hat{J}_\eps (\xi)}\hat{a}^-(\xi)\,d\xi-\frac{1}{\kam }\int_{|\xi|\geq \delta} \frac{\hat{J}_\eps (\xi)}{\kap -\hat{J}_\eps (\xi)}\hat{a}^-(\xi)\,d\xi.
\end{align*}
To estimate $p_{\geq \delta}^*(\eps)$, we verify firsly the following inequality:
\begin{equation}\label{eq:inequnif}
\kap - \hat{J}_\eps(\xi)> m-\hat{J}^*(\xi)\geq0, \qquad \xi\in\X.
\end{equation}
Namely, by \eqref{eq:compar}, \eqref{eq:defjeps}, the first inequality in \eqref{eq:inequnif} is equivalent to
\[
	\kap-m> \bigl(q^*-q^*(\eps)\bigr)\hat{a}^-(\xi), \qquad \xi\in\X,
\]
that is true since $|\hat{a}^-(\xi)|\leq \kam$, $\xi\in\X$, and $q^*-q^*(\eps)<q^*=\frac{\kap-m}{\kam}$. The second inequality in \eqref{eq:inequnif} is just \eqref{eq:lessm}. 

Next, since the function 
\[
m-\hat{J}^*(\xi)\geq m-\hat{J}^*(0)=0
\] 
is continuous in $\xi\in\X$, we conclude, cf.~\eqref{eq:pepsint}, that, for any $\delta>0$, there exists $\mu_\delta>0$, such that
\[
	m-\hat{J}^*(\xi)\geq \mu_\delta, \qquad |\xi|\geq \delta,
\]
and hence
\begin{equation}\label{eq:estextball}
 |p_{\geq \delta}^*(\eps)|\leq 
\frac{1}{\kam }  \frac{\kap}{\mu_\delta} \int_\X
\bigl\lvert\hat{a}^-(\xi)\bigr\rvert\,d\xi<\infty.
 \end{equation}

Next, by an analogue of \eqref{eq:dblestlemma} for $J_\eps$, we get, for small enough $\delta>0$,
\[
\kap-\hat{J}_\eps (\xi)\geq \mcr (\eps)-\hat{J}_\eps (\xi)\geq \frac{\pi^2}{\lVert B_\eps^{-1}\rVert^2}|\xi|^2,
\]
and hence,
\[
|p^*_{\leq \delta}(\eps)|\leq \const \int_{|\xi|\leq \delta} \frac{1}{|\xi|^2}\,d\xi<\infty \qquad \text{for } d\geq3.
\]
Combining the latter estimate with \eqref{eq:estextball}, we get that, for $d\geq3$, \eqref{eq:pstar0} holds by \eqref{eq:mcritconv} and the dominated convergence theorem.

Let now $d\leq 2$. By \eqref{eq:dblestlemma}, one can always choose $\delta_0>0$ small enough to ensure that, for $\delta<\delta_0$,
\[
\hat{a}^-(\xi)\geq  \kam - 3\pi^2\lVert B^-\rVert^2 \lvert \xi\rvert ^2>\frac{\kam}{2}>0, \qquad |\xi|\leq \delta.
\]
Then
\[
\hat{J}_\eps(\xi)=\hat{a}^+(\xi) - q^*(\eps)\hat{a}^-(\xi)\geq \hat{J}^*(\xi), \qquad |\xi|\leq \delta<\delta_0,
\]
and, possibly redefining $\delta_0$, we similarly get that 
\[
\hat{J}^*(\xi)\geq m - 3\pi^2\lVert B^*\rVert^2 \lvert \xi\rvert ^2>\frac{m}{2}>0, \qquad |\xi|\leq \delta<\delta_0.
\]
Next, by \eqref{eq:dblestlemma} applied to $a=J_\eps$, we get
\begin{align*}
	\kap-\hat{J}_\eps(\xi)&=\kap-\mcr(\eps)+\mcr(\eps)-\hat{J}_\eps(\xi)\\&\leq \kap-\mcr(\eps)+3\pi^2\|B_\eps\|^2 |\xi|^2\\
	&\leq \kap-\mcr(\eps)+3\pi^2\|B^+\|^2 |\xi|^2, \qquad |\xi|\leq \delta,
\end{align*}
where we used \eqref{eq:soineedthis}. Combining the previous inequalities, we get that, for a fixed $\delta<\delta_0$,
\begin{align*}
-p_{\leq \delta}^*(\eps)&=
\frac{1}{\kam }\int_{|\xi|\leq \delta}\frac{\hat{J}_\eps (\xi)}{\kap -\hat{J}_\eps (\xi)}\hat{a}^-(\xi)\,d\xi\\
&\geq \frac{m}{4}\int_{|\xi|\leq \delta}\frac{1}{\kap-\mcr(\eps)+3\pi^2\|B^+\|^2 |\xi|^2}\,d\xi.
\end{align*}

Therefore, for $d=2$, we get, by passing to polar coordinates, that
\[
	-p_{\leq \delta}^*(\eps)\geq c_1\log\biggl(1+\frac{c_2}{\kap-\mcr(\eps)}\biggr);
\] 
and, for $d=1$, we get that
\[
	-p_{\leq \delta}^*(\eps)\geq \frac{c_3}{\sqrt{\kap-\mcr(\eps)}} \arctan\frac{c_4}{\sqrt{\kap-\mcr(\eps)}},
\]
for certain $c_1,c_2,c_3,c_4>0$ (with $c_2,c_4$ depending on the fixed $\delta$).
 Since, by \eqref{eq:mcritconv}, $\kap-\mcr(\eps)\to0$ as $\eps\to0$, the statement is proved.
\end{proof}

\section{Asymptotics of the critical mortality: \texorpdfstring{$d\geq 3$}{d\textgreater=3}}\label{sec:asympd3}

We are going to reveal the asymptotic of $\mcr (\eps)$  (or, equivalently, $q^*(\eps)$) assuming that \eqref{eq:mcrchar} does hold. We start with the case $d\geq3$. 

If, additionally to \eqref{eq:newA1}--\eqref{eq:secmom} and \eqref{eq:mcrchar}, we assume that the limit $\lim\limits_{\eps\to0}\mcr (\eps)$ exists, then, by Proposition~\ref{prop:evidentlimit}, \eqref{eq:mcritconv} holds, and hence, by Proposition~\ref{eq:limitofpstar}, we get \eqref{eq:pstar0}. Then \eqref{eq:mcrchar} implies
\[
	\lim_{\eps\to0}\frac{q^*(\eps)}{\eps^d}=-\lim_{\eps\to0}p^*(\eps)-\lim_{\eps\to0}\frac{o(\eps^d)}{\eps^d}=- p^*(0)=\frac{1}{\kam }\int_\X \frac{\hat{a}^+(\xi)\hat{a}^-(\xi)}{\kap -\hat{a}^+(\xi)}\,d\xi.
\]
Since $q^*(\eps)>0$, we will get that $I\geq0$, where
\begin{equation}\label{eq:int}
I:=\int_\X\frac{\hat{a}^+(\xi)}{ \kap  -\hat{a}^+(\xi)}\hat{a}^-(\xi)\,d\xi.
\end{equation}

The first statement of the following theorem shows that one can replace the requiremnt about existence of the limit of $\mcr(\eps)$ by the continuity of the function
\begin{equation} \label{r1eps}
r(\eps):=\frac{o(\eps^d)}{\eps^d}
\end{equation}
in a neighbourhood of $0$, where $o(\eps^d)$ is from \eqref{eq:mcrchar}. Then, we reveal the next term of the asymptotic under additional smoothness of $r(\eps)$.

\begin{theorem}\label{thm:assympd3}
Let $d\geq3$  and \eqref{eq:newA1}--\eqref{eq:secmom}, \eqref{eq:mcrchar} hold. Let $r$ given by \eqref{r1eps} be  continuous for small $\eps>0$. Let $I\neq0$, where $I$ is given by \eqref{eq:int}.
Then $I>0$ and
\begin{equation}\label{eq:asymptdgeq3}
q^*(\eps)= \frac{I}{ \kam   } \eps^{d} +o(\eps^{d}).
\end{equation}
As a result,
\begin{equation}
\mcr (\eps)=\kap -\eps^d I +o(\eps^d).\label{eq:mcrdgeq3}
\end{equation}

If, additionally, $r(\eps)$ is continuously differentiable for small $\eps>0$ and if $r'(0):=\lim\limits_{\eps\to0+}r'(\eps)<\infty$, then it determines the next term of the asymptotics, namely, in \eqref{eq:asymptdgeq3}, \eqref{eq:mcrdgeq3}
\begin{equation}\label{eq:32235352dfgfdg}
o(\eps^{d})= - r'(0)\eps^{d+1} +o(\eps^{d+1}).
\end{equation}
\end{theorem}
\begin{proof}
We denote 
\begin{equation}\label{eq:4343regf536}
\la(\eps):=\eps^{-d} q^* (\eps).
\end{equation}
One can rewrite then \eqref{eq:mcrchar} as follows
\begin{equation}\label{eq:2345gsdx436}
\la(\eps)-\frac{1}{ \kam   }
	\int_\X\frac{\hat{a}^+(\xi)-\eps^d \la(\eps)\hat{a}^-(\xi)}{ \kap       -\hat{a}^+(\xi)+\eps^d \la(\eps)\hat{a}^-(\xi)}\hat{a}^-(\xi)\,d\xi+r(\eps)=0.
\end{equation} 

\emph{Step 1.} Note that, since $d\geq3$, we have $|I|<\infty$, by the arguments above. We set
\[
\la_3:=\frac{|I|}{\kam }\in(0,\infty),
\]
Let $\delta\in\bigl(0,\min\{\la_3,1\}\bigr)$ be such that, cf. \eqref{eq:qepsbdd},
\begin{equation}\label{eq:condondelta3}
(\la_3+\delta)\delta^d<\frac{\kap -m}{\kam },
\end{equation}
and let also $r(\eps)$, given by \eqref{r1eps}, be continuous on $(0,\delta)$. We set then
\begin{equation}\label{eq:sarqr33434}
r(0):=0=\lim\limits_{\eps\to0+}r(\eps), \qquad r(-\eps):=r(\eps), \quad \eps\in(0,\delta).
\end{equation}

For 
\[
(\la,\eps)\in E_\delta:=(\la_3- \delta,\la_3 +\delta)\times (- \delta,\delta),
\]
we consider the function
\[
	F(\la,\eps):=\frac{1}{ \kam   }
	\int_\X\frac{\hat{a}^+(\xi)-|\eps|^d \la\, \hat{a}^-(\xi)}{ \kap       -\hat{a}^+(\xi)+|\eps|^d \la\, \hat{a}^-(\xi)}\hat{a}^-(\xi)\,d\xi-\la \sgn(I) -r(\eps).
\]
Henceforth $\sgn(I)=1$ for $I>0$ and $\sgn(I)=-1$ for $I<0$.

For $(\la,\eps)\in E_\delta$ and $\delta$ as in \eqref{eq:condondelta3}, we have that $a^+(x)-\la|\eps|^d a^-(x)\geq 0$, $x\in\X$. Hence one can apply Proposition~\ref{eq:limitofpstar} for a (new) function $\mcr(\eps)$ such that $|\eps|^d \la=\frac{\kap -\mcr (\eps)}{\kam }$. It yields then
\begin{equation}\label{eq:Fla0}
\lim_{\eps\to0}F(\la,\eps)=F(\la,0)=\frac{1}{\kam }I-\la \sgn(I),
\end{equation}
and since $|I|\sgn(I)=I$, we get 
\begin{equation}\label{eq:evneed}
	F(\la_3,0)=0.
\end{equation}

\emph{Step 2.} By \eqref{eq:Fla0} and the dominated convergence theorem, $F$ is continuous on $E_\delta$, for small enough $\delta>0$. 
Prove that $\frac{\partial }{\partial \la}F(\la,\eps)$ 
is continuously differentiable on $E_\delta$, for small enough $\delta>0$. 
We have 
\begin{equation}\label{eq:sagewetew}
\frac{\partial }{\partial \la}F(\la,\eps)=-\frac{1}{ \kam   }
	\int_\X\frac{|\eps|^d \, \hat{a}^-(\xi)}{ \bigl(\kap       -\hat{a}^+(\xi)+|\eps|^d \la\, \hat{a}^-(\xi)\bigr)^2}\hat{a}^-(\xi)\,d\xi-\sgn(I),
\end{equation}
that is continuous for $(\la,\eps)\in E_\delta$, $\eps\neq0$. 

For $\eps=0$, $\la\in(\la_3- \delta, \la_3+\delta)$, we have by \eqref{eq:Fla0},
\begin{equation}\label{eq:partder}
\frac{\partial }{\partial \la}F(\la,0)=\lim_{h\to0}\frac{F(\la+h,0)-F(\la,0)}{h}=-\sgn(I)\neq0.
\end{equation}

By \eqref{eq:dblestlemma}, the denominator in \eqref{eq:sagewetew} behaves as $|\xi|^4$ near the origin that is integrable for $d\geq5$ only. In the latter case, using again \eqref{eq:dblestlemma} and the dominated convergence theorem, we get from \eqref{eq:sagewetew} that
\begin{equation}\label{eq:r23233524}
\lim_{\eps\to0} \frac{\partial }{\partial \la}F(\la,\eps)=-\sgn(I),
\end{equation}
hence, by \eqref{eq:partder}, $\frac{\partial }{\partial \la}F$ is continuous at $(\la,0)$ for all $\la\in(\la_3- \delta,\la_3 + \delta)$.

Let now $d=3,4$. Find $\lim_{\eps\to0} \frac{\partial }{\partial \la}F(\la,\eps)$. By the same arguments as for getting \eqref{eq:estextball}, we obtain that 
\[
	\int_{\X\setminus \Delta_\delta}\frac{\bigl(\hat{a}^-(\xi)\bigr)^2}{ \bigl(\kap       -\hat{a}^+(\xi)+|\eps|^d \la\, \hat{a}^-(\xi)\bigr)^2}\,d\xi<\infty,
\]
for any neighbourhood $\Delta_\delta$ of the origin, $0\in\Delta_\delta\subset\X$, of a positive Lebesgue measure.
Therefore, for any such $\Delta_\delta$ with small enough $\delta>0$, 
\[
\lim_{\eps\to0}
\frac{\partial }{\partial \la}F(\la,\eps)=-\frac{1}{ \kam   }
\lim_{\eps\to0} |\eps|^d h(\eps,\delta)-\sgn(I),
\]
where
\[
h(\eps, \delta):= 
	\int_{\Delta_\delta}
	\frac{\bigl(\hat{a}^-(\xi)\bigr)^2}{ \bigl(\kap       -\hat{a}^+(\xi)+|\eps|^d \la\, \hat{a}^-(\xi)\bigr)^2}\,d\xi.
\]
We define now
\begin{equation}\label{eq:Deltadelta}
	\Delta_\delta:=\{\xi\in\X : |B^+\xi|\leq \delta\},
\end{equation}
that is just an image of the ball $\{|\xi|\leq \delta\}$ under the mapping generated by a Hermitian positive definite matrix $D:=(B^+)^{-1}$. Note that $\det(D)>0$. Since $D $ generates a bounded continuous linear mapping on $\X$, $\Delta_\delta$ is a bounded neighbourhood of the origin. Then
\[
	h(\eps, \delta)= 
	\int_{\{|\xi|\leq \delta\}}
	\frac{\bigl(\hat{a}^-(D \xi)\bigr)^2}{ \bigl(\kap       -\hat{a}^+(D \xi)+|\eps|^d \la\, \hat{a}^-(D \xi)\bigr)^2}\,\det(D )d\xi\geq0.
\] 
The inequality \eqref{eq:dblest}, applied for $B=D $,
implies that $o(|D \xi|^2)=o(|\xi|^2)$ for $|\xi|\to0$. Then, by \eqref{eq:Texpcf}, for small enough $\delta>0$ and $|\xi|\leq \delta$, there exist $c_1,c_2>0$ 
\begin{equation}\label{eq:estorignew}
\begin{gathered}
c_1|\xi|^2\leq \kap-\hat{a}^+(D \xi)=2\pi^2|\xi|^2+o(|\xi|^2)\leq c_2 |\xi|^2,\\
\frac{\kam}{2}\leq \hat{a}^-(D \xi)=\kam+o(1)\leq \kam.
\end{gathered}
\end{equation}
Then, there exist $C_1,C_2,C_3,C_4>0$, such that
\begin{align*}
h(\eps,\delta)&\leq \int_{\{|\xi|\leq \delta\}}
	\frac{C_1}{ \bigl(|\xi|^2+|\eps|^d \la C_2\bigr)^2}\,d\xi
\\&\leq C_3\int_0^\delta \frac{r^{d-1}}{ \bigl(r^2+|\eps|^d \la C_2\bigr)^2}\,dr\leq C_3 \delta^{d-3}\int_0^\delta \frac{r^2 }{ \bigl(r^2+|\eps|^d \la C_2\bigr)^2}\,dr\\
&=C_4
\biggl(\frac{1}{\sqrt{|\eps|^d \la C_2}}
\arctan\Bigl( \frac{\delta}{\sqrt{|\eps|^d \la C_2}}\Bigr)-
\frac{\delta}{	(|\eps|^d \la C_2+\delta^2)}\biggr).
\end{align*}
As a result, $|\eps|^d h(\eps,\delta)\leq C_5|\eps|^{\frac{d}{2}}$ for some $C_5>0$, and hence \eqref{eq:r23233524} holds. Therefore, $\frac{\partial }{\partial \la}F(\la,\eps)$ is continuous at $(\la,0)$ as well.

\emph{Step 3.} As a result, $F(\la,\eps)$ is continuous on $E_\delta$ and continuously differentiable in $\la$ for small enough $\delta>0$. Since also \eqref{eq:evneed} holds, we conclude, by the implicit function theorem, that there exists (possibly smaller) $\delta>0$ and a unique function $\la=\la(\eps)$, $\eps\in(-\delta,\delta)$, such that 
$\la(0)=\la_3$ and
\begin{equation}\label{eq:23523tgmdsd}
F(\la(\eps),\eps)=0, \qquad \eps\in(-\delta,\delta);
\end{equation}
moreover, $\la(\eps)$ is continuous in $\eps\in(-\delta,\delta)$. The latter implies
\begin{equation*}
\la(\eps)=\la_3+o(1), \quad \eps\to0.
\end{equation*}
Since \eqref{eq:23523tgmdsd} implies \eqref{eq:2345gsdx436}, we get that
\[
	q^*(\eps)= \frac{|I|}{ \kam   } \eps^{d} +o(\eps^{d}).
\]
In particular, \eqref{eq:mcritconv} holds, and then, by \eqref{eq:pstar0}, we get from \eqref{eq:mcrchar} that $|I|=I$, i.e. $I>0$. Thus, one has
\eqref
{eq:asymptdgeq3} and, by \eqref{eq:qepsbdd}, we get also \eqref{eq:mcrdgeq3}.

\emph{Step 4.} Assume now, additionally, that $r(\eps)$ is continuously differentiable for $\eps\in(0,\delta)$ and that $r'(0):=\lim_{\eps\to0+}r'(\eps)<\infty$. Then \eqref{eq:sarqr33434} extends this to $\eps\in(-\delta,\delta)$. We have then that
\[
	\frac{\partial }{\partial \eps}F(\la,\eps)=-\frac{d}{ \kam   }
	\int_\X\frac{\eps|\eps|^{d-2} \,\la \, \hat{a}^-(\xi)}{ \bigl(\kap       -\hat{a}^+(\xi)+|\eps|^d \la\, \hat{a}^-(\xi)\bigr)^2}\hat{a}^-(\xi)\,d\xi-r'(\eps),
\]
is continuous for $(\la,\eps)\in E_\delta$, $\eps\neq0$. The same arguments as above show that, for $d\geq5$, $\frac{\partial }{\partial \eps}F$ is continuous at $(\la,0)$, $\la\in(\la_3- \delta,\la_3+\delta)$ with
\[
\frac{\partial }{\partial \eps}F(\la,0)=-r'(0).
\] 
For $d=3,4$, we have, by the same arguments as the above,
\begin{equation}\label{eq:adsqwwqr}
  \lim_{\eps\to0} \frac{\partial }{\partial \eps}F(\la,\eps)=C_6\lim_{\eps\to0}
\eps|\eps|^{d-2}O\bigl(|\eps|^{-\frac{d}{2}}\bigr)-r'(0)=-r'(0).
\end{equation}
as $d-1>\frac{d}{2}$ for $d\geq3$. Next,
\begin{gather*}
 \frac{\partial }{\partial \eps}F(\la,0)=\lim_{\eps\to0}\frac{F(\la,\eps)-F(\la,0)}{\eps}=\frac{1}{  \kam   }\lim_{\eps\to0}f(\la,\eps)-r'(0),
 \shortintertext{where}
\begin{aligned}
 f(\la,\eps):&=\frac{1}{ \eps}
	\int_\X\biggl(\frac{\hat{a}^+(\xi)-|\eps|^d \la\, \hat{a}^-(\xi)}{ \kap       -\hat{a}^+(\xi)+|\eps|^d \la\, \hat{a}^-(\xi)}-
\frac{\hat{a}^+(\xi)}{ \kap       -\hat{a}^+(\xi)}\biggr)
	\hat{a}^-(\xi)\,d\xi \\
	&=-\frac{\kap  \la|\eps|^d}{ \eps}
	\int_\X b(\la,\eps,\xi) \,d\xi,
\end{aligned}
\\\shortintertext{and}
b(\la,\eps,\xi):=\frac{\hat{a}^-(\xi)}{\bigl(\kap       -\hat{a}^+(\xi)+|\eps|^d \la\, \hat{a}^-(\xi)\bigr)\bigl( \kap       -\hat{a}^+(\xi)\bigr)}.
\end{gather*}
Since $d>1$, by re-choosing $\delta>0$, we get, similarly to the arguments above, that
\begin{align*}
 \lim_{\eps\to0}f(\la,\eps)=-\kap  \la\lim_{\eps\to0}\frac{|\eps|^d}{ \eps}
 \int_{\Delta_\delta} b(\la,\eps,\xi)\,d\xi,
 \end{align*}
where $\Delta_\delta$ is given by \eqref{eq:Deltadelta}.
By the change of variables and \eqref{eq:estorignew}, we have, for small enough $\delta>0$, and some constants $c_1,c_2,c_3,c_4>0$,
\begin{align*}
&\quad \biggl\lvert \frac{|\eps|^d}{ \eps}
 \int_{\Delta_\delta} b(\la,\eps,\xi)\,d\xi \biggr\rvert \\&\leq |\eps|^{d-1}\kam \det (D) \int_{\{|\xi|\leq d\}}
  \frac{1}{\bigl(c_1|\xi|^2+\frac{\kam}{2}|\eps|^d \la \bigr)\bigl( c_1|\xi|^2\bigr)}\,d\xi\\
  &\leq |\eps|^{d-1} c_2 \int_{\{|\xi|\leq d\}}
  \frac{1}{\bigl(|\xi|^2+c_3\la |\eps|^d   \bigr)|\xi|^2}\,d\xi=|\eps|^{d-1} c_4 \int_0^\delta
  \frac{r^{d-1}}{\bigl(r^2+c_3\la |\eps|^d   \bigr)r^2}\,d\xi\\
   &\leq |\eps|^{d-1} c_4 \delta^{d-3}\int_0^\delta
  \frac{1}{\bigl(r^2+c_3\la |\eps|^d   \bigr)}\,d\xi\\
  &=c_4\delta^{d-3}|\eps|^{d-1}\frac{1}{\sqrt{c_3 \la|\eps|^d}}\arctan\frac{\delta}{\sqrt{c_3 \la|\eps|^d}}\to0, \qquad \eps\to0.
\end{align*}
Therefore, for $d=3,4$, we also have
\[
\frac{\partial }{\partial \eps}F(\la,0)=-r'(0)
\]
and hence $\frac{\partial }{\partial \eps}F$ is continuous at $(\la,0)$ for $\la\in(\la_3- \delta,\la_3+\delta)$. 

As a result, both partial derivatives $\frac{\partial }{\partial \la}F$ and $\frac{\partial }{\partial \eps}F$ are continuous on $E_\delta$, hence $F$ is continuously differentiable on $E_\delta$.

Then, the implicit function theorem ensures that the unique function $\la=\la(\eps)$, $\eps\in(-\delta,\delta)$, such that $\la(0)=\la_3$ and \eqref{eq:23523tgmdsd} holds, is continuously differentiable in $\eps$. Differentiating \eqref{eq:23523tgmdsd} in $\eps$, we get
\[
\la'(\eps)\frac{\partial }{\partial \la}F(\la(\eps),\eps)+\frac{\partial }{\partial \eps}F(\la(\eps),\eps)=0,
\]
and hence, by \eqref{eq:partder}, \eqref{eq:adsqwwqr}
\begin{align*}
\la'(0)&=\frac{\partial }{\partial \eps}F(\la(\eps),\eps)\biggr\rvert_{\eps=0}=-r'(0).
\end{align*}
As a result,
\[
\la(\eps)= \la_3 - r'(0)\eps +o(\eps),
\]
that, by \eqref{eq:4343regf536} implies \eqref{eq:32235352dfgfdg}. 
\end{proof}

\begin{remark}\label{eq:remindecoef}
Note that $\la_3=\frac{I}{\kap}$, where $I$ is given by \eqref{eq:int}, does not actually depend on $\kap$ nor $\kam$.
\end{remark}

\section{Asymptotics of the critical mortality: \texorpdfstring{$d= 2$}{d=2}}\label{sec:asympd2}

Let $d=2$ and \eqref{eq:mcrchar} hold. 
We start with some heuristic arguments. Let, additionally, \eqref{eq:mcritconv} hold; then, by Proposition~\ref{eq:limitofpstar},  $\lim_{\eps\to0} p^*(\eps)=-\infty$. 
By~\eqref{eq:mcrchar}, \eqref{eq:mcritconv}, we can expect then that
\begin{equation}\label{eq:anz}
q^*(\eps)\approx\eps^2 l(\eps),
\end{equation}
where $l(\eps)\sim p^*(\eps)$, $\eps\to0$, and therefore,
\begin{equation}\label{eq:heur}
\lim_{\eps\to0} l(\eps)=\infty, \qquad \lim_{\eps\to0} \eps^2 l(\eps)=0.
\end{equation}
Then, by \eqref{eq:pepsint}, the anzatz \eqref{eq:anz} implies 
\begin{align*}
l(\eps)&\sim \frac{1}{ \kam   }
	\int_{\R^2}\frac{\hat{a}^+(\xi)-\eps^2 l(\eps)\, \hat{a}^-(\xi)}{ \kap       -\hat{a}^+(\xi)+\eps^2 l(\eps)\, \hat{a}^-(\xi)}\hat{a}^-(\xi)\,d\xi\\&=\frac{\kap }{\kam }
\int_{\R^2} \frac{1}{\kap -\bigl(\hat{a}^+ (\xi)- \eps^2 l(\eps) \hat{a}^- (\xi)\bigr)}\hat{a}^-(\xi)\,d\xi-\frac{1}{\kam }
\int_{\R^2} \hat{a}^-(\xi)\,d\xi.
\end{align*}

By the same arguments as above, the singularity of the latter expression as $\eps\to0$ is fully determined by the integral
\[
\sigma(\delta,\eps)	:=
\frac{\kap}{\kam }
\int_{\Delta_\delta} \frac{1}{\kap -\bigl(\hat{a}^+ (\xi)- \eps^2 l(\eps) \hat{a}^- (\xi)\bigr)}\hat{a}^-(\xi)\,d\xi
\]
for small enough $\delta>0$, where $\Delta_\delta$ is given by \eqref{eq:Deltadelta}. By \eqref{eq:estorig} and the change of variables as above, 
\begin{align*}
	\sigma(\delta,\eps)\sim\const \cdot
\int_{|\xi|\leq \delta} \frac{1}{|\xi|^2+ \kam  \eps^2 l(\eps)}\,d\xi=\const \cdot\int_0^\delta  \frac{r}{r^2+ \kam  \eps^2 l(\eps)}\,dr.
\end{align*}
Integrating, we conclude that, heuristically, for some $c_1,c_2>0$,
\begin{align*}
 l(\eps)&\sim c_1 \log\biggl(1+\frac{c_2}{\eps^2 l(\eps)}\biggr)\sim c_1\log \frac{c_2}{\eps^2 l(\eps)}, 
\end{align*}
by \eqref{eq:heur}, i.e. $\frac{l(\eps)}{c_1} e^{\frac{l(\eps)}{c_1}}\sim \frac{c_2}{c_1\eps^2}$. 
Therefore,
\[
	l(\eps)\sim c_1 W\biggl( \frac{c_2}{c_1\eps^2} \biggr),
\]
where $W(z)$, $z>0$, is the unique solution to the equation $ye^y=z$, $y>0$, (the principal branch of) of the so-called Lambert~$W$ function. It is well-known that
\begin{equation*}
W(z)= \log z - \log \log z+o(1), \qquad z\to+\infty.
\end{equation*}
Therefore, 
\[
	W\biggl( \frac{c_2}{c_1\eps^2} \biggr)=-2\log\eps -\log(-\log\eps)+O(1),
\]
in other words, two leading terms of the asymptotics of $q^*(\eps)\approx c'\eps^2 W(c\eps^{-2})$ depend on $c'$ but not on $c$. This gives a hint to define $\la(\eps)$ in the proof of the following theorem.

\begin{theorem}\label{thm:assympd2}
Let $d=2$ and \eqref{eq:newA1}--\eqref{eq:secmom} hold. Let \eqref{eq:mcrchar} hold with $o(\eps^2)$ such that the function
\begin{equation}
r(\eps):=\frac{o(\eps^2)}{\eps^2 W(\eps^{-2})} \label{r2eps}
\end{equation}
is continuous for small $\eps>0$.  Then 
\begin{equation}
\label{eq:asymptford2}
q^*(\eps)=\la_2\eps^2W(\eps^{-2})+o\bigl(\eps^2 W(\eps^{-2})\bigr),
\end{equation}
where
\begin{equation}\label{eq:la2}
\la_2:=\dfrac{\kap }{2\pi \sqrt{a_{11}a_{22}-a_{12}^2}},
\end{equation}
and $a_{ij}$, $1\leq i,j\leq 2$, are given by \eqref{eq:defofaplus}.
\end{theorem}
\begin{remark}\label{eq:remindecoef2}
Note that $\la_2$ does not actually depend on $\kap$ (cf. Remark~\ref{eq:remindecoef}). 
\end{remark}
\begin{proof}
We define $\la(\eps)>0$, $\eps\in(0,1)$, through the equality
\begin{equation}\label{eq:expofqlog}
q^*(\eps)=\la(\eps)\eps^2W(\eps^{-2}).
\end{equation}
One can rewrite then \eqref{eq:mcrchar} as follows
\[
	\la(\eps)-\frac{1}{ \kam  W(\eps^{-2})}
	\int_{\R^2}\frac{\hat{a}^+(\xi)-  \eps^2W(\eps^{-2})\la(\eps)\,   \hat{a}^-(\xi)}{ \kap       -\hat{a}^+(\xi)+\eps^2W(\eps^{-2}) \la(\eps)\, \hat{a}^-(\xi))}\hat{a}^-(\xi)\,d\xi+r(\eps)=0,
\]
where $r(\eps)\to0$, $\eps\to0$, is given by \eqref{r2eps}

Let $\delta\in\bigl(0,\min\{\la_2,1\}\bigr)$ be such that, cf. \eqref{eq:qepsbdd}, \eqref{eq:condondelta3},
\begin{equation}\label{eq:condondelta}
(\la_2+\delta)\delta^2 W(\delta^{-2})<\frac{\kap -m}{\kam },
\end{equation}
and let also $r(\eps)$ be continuous on $(0,\delta)$. Note that the function in the left hand side of \eqref{eq:condondelta} is increasing in $\delta>0$. We define then $r$ on $(-\delta,0]$ as in \eqref{eq:sarqr33434}.
For $\la\in (\la_2- \delta,\la_2 +\delta)$, $\eps\in (- \delta,\delta)$, $\eps\neq0$,
we consider the function
\begin{align}
F(\la,\eps):&=\frac{1}{ \kam   W(\eps^{-2})}
	\int_{\R^2}\frac{\hat{a}^+(\xi)-\la  \eps^2 W(\eps^{-2}) \, \hat{a}^-(\xi)}{ \kap       -\hat{a}^+(\xi)+\la  \eps^2 W(\eps^{-2}) \, \hat{a}^-(\xi)}\hat{a}^-(\xi)\,d\xi-\la -r(\eps)\notag\\
	&=\frac{\kap }{ \kam    W(\eps^{-2})}
	\int_{\R^2}\frac{1}{ \kap       -\hat{a}^+(\xi)+\la  \eps^2 W(\eps^{-2})\, \hat{a}^-(\xi)}\hat{a}^-(\xi)\,d\xi \notag\\&\quad -\frac{1}{ \kam  W(\eps^{-2}) }	\int_{\R^2} \hat{a}^-(\xi)\,d\xi-\la -r(\eps).\label{eq:defFleps2}
\end{align}
Clearly, $F$ is continuous for $\la\in (\la_2- \delta,\la_2 +\delta)$, $\eps\in (- \delta,\delta)$, $\eps\neq0$.

By the same arguments as above,
\begin{gather}\label{eq:expchm}
 \lim_{\eps\to0}F(\la,\eps)=-\la+\lim_{\eps\to0}\frac{1}{W(\eps^{-2})}\sigma(\eps,\delta,\la),
 \\\intertext{where, for a small enough $\delta>0$,}
\sigma(\eps,\delta,\la):=\frac{\kap }{\kam }\int_{\Delta_\delta}\frac{1}{ \kap       -\hat{a}^+(\xi)+\la  \eps^2 W(\eps^{-2})\, \hat{a}^-(\xi)}\hat{a}^-(\xi)\,d\xi,\notag
\end{gather}
and $\Delta_\delta$ is given by \eqref{eq:Deltadelta}.

Recall that the inequality \eqref{eq:dblest}, applied for $B=D $,
implies that $o(|D \xi|^2)=o(|\xi|^2)$ for $|\xi|\to0$. Then, cf.~\eqref{eq:estorignew}, by \eqref{eq:Texpcf}, for any $\rho\in(0,1)$ there exists $\delta_\rho>0$ small enough such that, for $|\xi|\leq\delta<\delta_\rho$, 
\begin{equation}\label{eq:estorig}
\begin{gathered}
2\pi^2(1- \rho)|\xi|^2\leq \kap-\hat{a}^+(D \xi)=2\pi^2|\xi|^2+o(|\xi|^2)\leq 2\pi^2(1+ \rho) |\xi|^2,\\
(1-\rho)\kam\leq \hat{a}^-(D \xi)=\kam+o(1)\leq \kam.
\end{gathered}
\end{equation}

By change of variables and \eqref{eq:estorig}, for each $\rho\in(0,1)$, there exists $\delta_\rho<\la_2$, such that, for all $\delta\in(0,\delta_\rho)$,
\begin{equation}\label{eq:dblestsigma}
\frac{1}{1-\rho}\tau(\eps,\delta,\la)\geq \sigma(\eps,\delta,\la)\geq \frac{1-\rho}{1+\rho} \tau(\eps,\delta,\la),
\end{equation}
where, for $D=(B^+)^{-1}$,
\begin{align*}
\tau(\eps,\delta,\la):&= \kap \det(D )\int_{|\xi|\leq \delta}\frac{1}{ 2\pi^2 |\xi|^2+\kam \la  \eps^2 W(\eps^{-2})}\,d\xi\\&=2\kap \pi \det(D ) \int_0^\delta \frac{1}{ 2\pi^2 r^2+\kam \la  \eps^2 W(\eps^{-2})}\,r\, dr\\
&=\frac{\kap }{2\pi} \det(D ) \log\biggl(1 + \frac{2\pi^2\delta^2}{\kam \la  \eps^2 W(\eps^{-2})} \biggr).
\end{align*}
Note that $W(\eps^{-2})e^{W(\eps^{-2})}=\eps^{-2}$, i.e.
$\eps^2 W(\eps^{-2})=e^{-W(\eps^{-2})}$. Set $R:=W(\eps^{-2})\to+\infty$, $\eps\to0$. Then
\begin{align*}
&\quad \frac{1}{W(\eps^{-2})}\tau(\eps,\delta,\la)
	=\frac{ \kap \det(D ) }{2\pi R}\log\biggl(1 + \frac{2\pi^2\delta^2}{\kam \la  } e^{R}\biggr)\\
	&=\frac{\kap  \det(D ) }{2\pi R}\Biggl(\log\frac{2\pi^2\delta^2}{\kam \la  }+R+\log\biggl(1 + \frac{\kam \la  }{2\pi^2\delta^2} e^{-R}\biggr)\Biggr)\to \frac{ \kap \det(D ) }{2\pi},
\end{align*}
as $R\to+\infty$, i.e. as $\eps\to0$. Combining this with \eqref{eq:expchm} and \eqref{eq:dblestsigma}, we conclude that, for each $\rho\in(0,1)$,
\[
 \lim_{\eps\to0}F(\la,\eps)+\la\in\biggl( \frac{\kap \det(D )}{2\pi}\frac{1-\rho}{1+\rho}, \frac{\kap \det(D )}{2\pi} \frac{1}{1-\rho}\biggr).
\]
By sending $\rho$ to $0$, we get
\[
\lim_{\eps\to0}F(\la,\eps)=-\la+\frac{\kap \det(D )}{2\pi} =-\la+\frac{\kap }{2\pi\det(B^+)}=-\la+\la_2,
\]
where $\la_2$ is given by \eqref{eq:la2}, since $(B^+)^2=A^+$ implies $\det(B^+)=\sqrt{\det(A^+)}$.

Therefore, if we set
\[
F(\la,0):=\la_2-\la, \qquad\la\in(\la_2- \delta,\la_2+\delta),
\]
then $F(\la,\eps)$ becomes a continuous function on
\[
E_\delta:=(\la_2- \delta,\la_2 +\delta)\times (- \delta,\delta)
\]
with a small enough $\delta\in(0,\la_2)$. Moreover, $F(\la_2,0)=0$. Next, since
\[
\frac{F(\la+h,0)-F(\la,0)}{h}=-1, \qquad \la,\la+h\in(\la_2- \delta,\la_2+\delta),
\]
we have 
\begin{equation*}
\frac{\partial}{\partial \la}F(\la,0)=-1\neq0, \qquad \la
\in(\la_2- \delta,\la_2+\delta).
\end{equation*}
Next, for $(\la,\eps)\in E_\delta$, $\eps\neq0$, we have, by \eqref{eq:defFleps2},
\begin{align*}
\frac{\partial}{\partial \la}F(\la,\eps)&=-1- \frac{1}{ \kam   W(\eps^{-2})}
	\int_{\R^2}\frac{  \eps^2 W(\eps^{-2})\, \hat{a}^-(\xi)}{ \bigl(\kap       -\hat{a}^+(\xi)+\la  \eps^2 W(\eps^{-2}) \, \hat{a}^-(\xi)\bigr)^2}\hat{a}^-(\xi)\,d\xi\\
	&=-1- \frac{ \eps^2 }{ \kam   }
	\int_{\R^2}\frac{ \bigl(\hat{a}^-(\xi)\bigr)^2}{ \bigl(\kap       -\hat{a}^+(\xi)+\la  \eps^2 W(\eps^{-2})\, \hat{a}^-(\xi)\bigr)^2}\,d\xi.
\end{align*}
By the same arguments as above,
\begin{gather*}
\lim_{\eps\to0}
\frac{\partial }{\partial \la}F(\la,\eps)=-1-\frac{1}{ \kam   }
\lim_{\eps\to0}  \eps^2  h(\eps,\delta),
\shortintertext{where}
h(\eps, \delta)= 
	\int_{\Delta_\delta}
	\frac{\bigl(\hat{a}^-(\xi)\bigr)^2}{ \bigl(|B^+\xi|^2+o(|\xi|^2)+\la  \eps^2 W(\eps^{-2}) (\kam +o(1))\bigr)^2}\,d\xi,
\end{gather*}
where $\Delta_\delta$ is given by \eqref{eq:Deltadelta}. By change and variables and \eqref{eq:estorignew},  we get that, for some $C_1,C_2,C_3>0$ and for small enough $\delta>0$,
\begin{align*}
0<h(\eps,\delta)&\leq \int_{\{|\xi|\leq \delta\}}
	\frac{C_1}{ \bigl(|\xi|^2+ \la\eps^2 W(\eps^{-2}) C_2\bigr)^2}\,d\xi\\
&\leq C_3\int_0^\delta \frac{r }{ \bigl(r^2+ \la\eps^2   W(\eps^{-2})C_2\bigr)^2}\,dr\\
&=\frac{C_3}{2C_2}\frac{\delta^2}{ \la \eps^2  W(\eps^{-2}) \bigl(\delta^2+ \la\eps^2  W(\eps^{-2}) C_2\bigr)},
\end{align*}
and hence $ \eps^2 h(\eps,\delta)\to0$, $\eps\to0$, that yields
\[
\lim_{\eps\to0}
\frac{\partial }{\partial \la}F(\la,\eps)=-1,
\]
and therefore, $\frac{\partial }{\partial \la}F$ is continuous on $E_\delta$.

Again, the implicit function theorem states that there exists a unique function $\la=\la(\eps)$, $\eps\in(-\delta,\delta)$ (with, possibly, smaller $\delta$), such that
$\la(0)=\la_2$ and
\[
F(\la(\eps),\eps)=0, \qquad\eps\in(-\delta,\delta);
\] 
moreover, $\la(\eps)$ is \emph{continuous} in $\eps\in(-\delta,\delta)$.
Therefore, $\la(\eps)=\la_2+o(1)$, $\eps\to0$; hence, from \eqref{eq:expofqlog} and \eqref{eq:qepsbdd}, we get \eqref{eq:asymptford2}.
\end{proof}

\begin{corollary}
If function $a^+$ in Theorem~\ref{thm:assympd2} is radially symmetric, i.e. $a^+(x)=b^+(|x|)$ for some $b^+:\R\to\R$, then $a_{12}=a_{21}=0$ and 
\[
	a_{11}=a_{22}=\int_{\R^2} x_1^2 a^+(x)\,dx=\frac{1}{2}\int_{\R^2}|x|^2 a^+(x)\,dx,
\] 
so that
\begin{equation*}
\la_2=\frac{\kap}{\pi \int_{\R^2}|x|^2 a^+(x)\,dx}.
\end{equation*}
\end{corollary}

\begin{remark}\label{rem:nondif2}
It can be checked that $F(\la,\eps)$ defined by \eqref{eq:defFleps2} is not continuously differentiable in $\eps$ at $\eps=0$ (even if we assume that $r$ is); hence, in general, one can not expect that $\la(\eps)$ is continuously differentiable at $\eps=0$. Hence, the question about the next  term of the assymptotic in \eqref{eq:asymptford2} remains open.
\end{remark}

\section{Asymptotics of the critical mortality: \texorpdfstring{$d=1$}{d=1}}\label{sec:asympd1}

Let $d=1$ and \eqref{eq:mcrchar} hold. 
We firstly proceed again heuristically. Similarly to the arguments at the beginning of Section~\ref{sec:asympd2}, if \eqref{eq:mcritconv} holds then we may expect, for $\eps\to0$,
\begin{gather}\label{eq:anz1}
q^*(\eps)\approx\eps l(\eps), \\
\shortintertext{where}
l(\eps)\sim p^*(\eps)\to\infty, \qquad \eps l(\eps)\to 0.\notag 
\end{gather}
Then, by \eqref{eq:pepsint}, the ansatz \eqref{eq:anz1} implies 
\begin{align*}
l(\eps)&\sim \frac{1}{ \kam   }
	\int_{\R}\frac{\hat{a}^+(\xi)-\eps  l(\eps)\, \hat{a}^-(\xi)}{ \kap       -\hat{a}^+(\xi)+\eps l(\eps)\, \hat{a}^-(\xi)}\hat{a}^-(\xi)\,d\xi\\&=\frac{1}{\kap \kam }
\int_{\R} \hat{a}^-(\xi)\,d\xi-
\frac{1}{\kam }
\int_{\R} \frac{1}{\kap -\bigl(\hat{a}^+ (\xi)- \eps l(\eps) \hat{a}^- (\xi)\bigr)}\hat{a}^-(\xi)\,d\xi.
\end{align*}

By the same arguments as above, the singularity of the latter expression as $\eps\to0$ is fully determined by the integral
\begin{align*}
&\quad
\frac{1}{\kam }
\int_{-\delta}^\delta \frac{1}{\kap -\bigl(\hat{a}^+ (\xi)- \eps l(\eps) \hat{a}^- (\xi)\bigr)}\hat{a}^-(\xi)\,d\xi\\
&\sim c_1 \cdot\int_0^\delta  \frac{1}{r^2+ c_2  \eps l(\eps)}\,dr=\frac{c_3}{\sqrt{\eps l(\eps)}}\arctan \frac{\delta}{\sqrt{c_2  \eps l(\eps)}},
\end{align*}
for small enough $\delta>0$ and some $c_1,c_2,c_3>0$; here we used \eqref{eq:Texpcf}. As a result, heuristically,
\[
	l(\eps)\sqrt{\eps l(\eps)}\approx \const,
\]
and hence $l(\eps)\approx \const \eps^{-\frac13}$, $\eps\to0$. Again, it gives us a hint to define $\la(\eps)$ in the proof of the following theorem.

\begin{theorem}\label{thm:assympd1}
Let $d=1$ and \eqref{eq:newA1}--\eqref{eq:secmom} hold. Let \eqref{eq:mcrchar} hold with $o(\eps)$ such that the function
\begin{equation}
r(\eps):=\eps^{-\frac{2}{3}}o(\eps) \label{r1eps1}
\end{equation}
is continuous for small $\eps>0$.  Then 
\begin{equation}
\label{eq:asymptford1}
q^*(\eps)=\la_1\eps^{\frac{2}{3}}+o\bigl(\eps^{\frac{2}{3}}\bigr),
\end{equation}
where
\begin{equation*}
\la_1:=\Biggl(\dfrac{(\kap)^2}{ 2\kam\displaystyle\int_\R x^2 a^+(x)\,dx}\Biggr)^{\frac13}.
\end{equation*}
\end{theorem}
\begin{remark}
Note that, in contrast to the cases $d\geq 3$ and $d=2$, cf.~Remarks~\ref{eq:remindecoef}, \ref{eq:remindecoef2}, $\la_1$ depends effectively on (the ratio of) $\kap$ and $\kam$.
\end{remark}
\begin{proof}
We set $\la(\eps):=q^*(\eps)\eps^{-\frac{2}{3}}$ for $\eps\in(0,1)$,
and then rewrite  \eqref{eq:mcrchar} as follows
\[
	\la(\eps)-\frac{\eps^{\frac{1}{3}}}{ \kam  }
	\int_\R\frac{\hat{a}^+(\xi)-  \eps^{\frac{2}{3}}\la(\eps)\,   \hat{a}^-(\xi)}{ \kap       -\hat{a}^+(\xi)+\eps^{\frac{2}{3}}\la(\eps)\, \hat{a}^-(\xi))}\hat{a}^-(\xi)\,d\xi+r(\eps)=0,
\]
where $r(\eps)\to0$, $\eps\to0$, is given by \eqref{r1eps1}

Let $\delta\in\bigl(0,\min\{\la_1,1\}\bigr)$ be such that, cf. \eqref{eq:qepsbdd}, \eqref{eq:condondelta3}, \eqref{eq:condondelta},
\begin{equation*}
(\la_1+\delta)\delta^{\frac{2}{3}}<\frac{\kap -m}{\kam },
\end{equation*}
and let also $r(\eps)$ be continuous on $(0,\delta)$. We define then $r$ on $(-\delta,0]$ as in \eqref{eq:sarqr33434}.
For $\la\in (\la_1- \delta,\la_1 +\delta)$, $\eps\in (- \delta,\delta)$, $\eps\neq0$,
we consider the function
\begin{align}
F(\la,\eps):&=\frac{\eps^{\frac{1}{3}}}{ \kam  }
	\int_{\R^2}\frac{\hat{a}^+(\xi)-\la \eps^{\frac{2}{3}} \, \hat{a}^-(\xi)}{ \kap       -\hat{a}^+(\xi)+\la  \eps^{\frac{2}{3}} \, \hat{a}^-(\xi)}\hat{a}^-(\xi)\,d\xi-\la -r(\eps)\notag\\
	&=\frac{\kap\eps^{\frac{1}{3}}}{\kam}
	\int_{\R^2}\frac{1}{ \kap       -\hat{a}^+(\xi)+\la  \eps^{\frac{2}{3}}\, \hat{a}^-(\xi)}\hat{a}^-(\xi)\,d\xi \notag\\&\quad -\frac{\eps^{\frac{1}{3}}}{  \kam }	\int_{\R^2} \hat{a}^-(\xi)\,d\xi-\la -r(\eps).\label{eq:defFleps1}
\end{align}
Clearly, $F$ is continuous for $\la\in (\la_1- \delta,\la_1+\delta)$, $\eps\in (- \delta,\delta)$, $\eps\neq0$.

Let 
\[
B:=\int_\R |x|^2 a^+(x)\,dx=2\int_0^\infty x^2 a^+(x)\,dx.
\]
By \eqref{eq:Texpcf} and the same arguments as in the proof of Theorem~\ref{thm:assympd2},
\begin{align*}
\lim_{\eps\to0}F(\la,\eps)&=-\la+\frac{\kap}{\kam }\lim_{\eps\to0}\eps^{\frac{1}{3}}\int_{-\delta}^\delta\frac{1}{ \kap       -\hat{a}^+(\xi)+\la  \eps^{\frac{2}{3}}\, \hat{a}^-(\xi)}\hat{a}^-(\xi)\,d\xi\\
&=-\la+2\kap\lim_{\eps\to0}\eps^{\frac{1}{3}}\int_{0}^\delta\frac{1}{ 2\pi^2 B r^2+\kam\la  \eps^{\frac{2}{3}}}\,dr\\
&= - \la +\frac{\sqrt{2}\kap}{\pi\sqrt{\la \kam B}}\lim_{\eps\to0}
\arctan\frac{\sqrt{2B}\pi \delta}{\sqrt{\kam  \la}\eps^{\frac13}}\\
&= - \la +\frac{\kap}{\sqrt{2\la \kam B}}.
\end{align*}

Therefore, if we set
\[
F(\la,0):= - \la +\frac{\kap}{\sqrt{2\la \kam B}}, \qquad\la\in(\la_1- \delta,\la_1+\delta),
\]
then $F(\la,\eps)$ becomes a continuous function on
\[
E_\delta:=(\la_1- \delta,\la_1 +\delta)\times (- \delta,\delta)
\]
with a small enough $\delta\in(0,\la_1)$. Moreover, it is straightforward to check that 
\[
F(\la_1,0)=0.
\]

For $\la\in(\la_1- \delta,\la_1+\delta)$, we have
\[
\frac{\partial}{\partial \la}F(\la,0)=\lim_{h\to0}\frac{F(\la+h,0)-F(\la,0)}{h}=-1-\frac{\kap}{2\sqrt{2 \kam B}}\la^{-\frac{3}{2}}< 0,
\]
and also
\[
	\frac{\partial}{\partial \la}F(\la_1,0)=
	-1-\frac{\kap}{2\sqrt{2 \kam B}}
	\Bigl((\kap)^{\frac23}(2\kam B)^{-\frac13}\Bigr)^{-\frac{3}{2}}=-\frac32\neq0.
\]

Next, for $(\la,\eps)\in E_\delta$, $\eps\neq0$, we have, by \eqref{eq:defFleps1},
\[
\frac{\partial}{\partial \la}F(\la,\eps)=-1- \frac{ \eps }{ \kam   }
	\int_{\R}\frac{ \bigl(\hat{a}^-(\xi)\bigr)^2}{ \bigl(\kap       -\hat{a}^+(\xi)+\la  \eps^{\frac{2}{3}}\, \hat{a}^-(\xi)\bigr)^2}\,d\xi.
\]
By the same arguments as above,
\begin{align*}
 \lim_{\eps\to0}
\frac{\partial }{\partial \la}F(\la,\eps)&
=-1-\frac{\kap}{ \kam   }
\lim_{\eps\to0}  \eps  
	\int_{-\delta}^\delta
	\frac{\bigl(\hat{a}^-(\xi)\bigr)^2}{  \bigl(\kap       -\hat{a}^+(\xi)+\la  \eps^{\frac{2}{3}}\, \hat{a}^-(\xi)\bigr)^2}\,d\xi\\
	&=	-1-\kap\kam
\lim_{\eps\to0}  \eps  
	\int_{-\delta}^\delta
	\frac{1}{  \bigl(2\pi^2 B|\xi|^2+\la \kam  \eps^{\frac{2}{3}} \bigr)^2}\,d\xi\\
	\intertext{and by straightforward integration, one gets}
	&= -1- \kap
\lim_{\eps\to0} 
\frac{\delta \eps^{\frac13}}{2\pi^2 B \delta^2 \kam \la +\eps^{\frac23}(\kam)^2\la^2}\\&\quad 
-\kap
\lim_{\eps\to0}\frac{1}{\pi \sqrt{2B\kam}\la^{\frac32}}\arctan\biggl(\sqrt{\frac{2 B}{\kam\la}}\pi \delta \eps^{-\frac13}\biggr)
\Biggr)\\
&=-1-\frac{\kap}{2\sqrt{2B\kam}}\la^{-\frac32}=\frac{\partial}{\partial \la}F(\la,0).
\end{align*}
Therefore, $\frac{\partial }{\partial \la}F$ is continuous on $E_\delta$.

Again, the implicit function theorem states that there exists a unique continuous function $\la=\la(\eps)$  such that $\la(0)=\la_1$ and $F(\la(\eps),\eps)=0$, $\eps\in(-\delta,\delta)$ $\eps\in(-\delta,\delta)$ (with, possibly, smaller $\delta$). 
Therefore, $\la(\eps)=\la_1+o(1)$, $\eps\to0$, that yields \eqref{eq:asymptford1}.
\end{proof}

\begin{remark}
Similarly to the case $d=2$, see Remark~\ref{rem:nondif2}, function $F(\la,\eps)$ defined by \eqref{eq:defFleps1} is not continuously differentiable in $\eps$ at $\eps=0$, hence the next  term of the assymptotic in \eqref{eq:asymptford1} remains an open problem.
\end{remark}

\section*{Acknowledgements}%
\addcontentsline{toc}{section}{\numberline{}Acknowledgements}%

I thank Otso Ovaskainen and Panu Somervuo for discussions that motivated the analyses presented in this paper.

\def\cprime{$'$}

\end{document}